\documentclass[11pt,english,a4paper]{smfart}

\usepackage[francais]{babel}
\usepackage{psfrag}
 \usepackage{vaucanson-g}
\usepackage{graphicx}
\usepackage{ae,amsfonts,euscript,enumerate}
\usepackage[cm]{aeguill}
\renewcommand{\theequation}{\thesection.\arabic{equation}}

\newtheorem{thm}{Theorem}[section]

\newtheorem{prop}[thm]{Proposition}
\newtheorem{lem}[thm]{Lemma}

\newtheorem{prob}[thm]{Problem}
\theoremstyle{definition}
\newtheorem{exam}[thm]{Example}
\newtheorem{defn}[thm]{Definition}

\newtheorem{rem}[thm]{Remark}
\numberwithin{equation}{section}
\numberwithin{figure}{section}
\def\Dio{{\rm dio}}
\usepackage[utf8x]{inputenc} 
\newcommand{\xvdash}[1]{%
  \vdash^{\mkern-10mu\scriptscriptstyle\rule[-.9ex]{0pt}{0pt}#1}%
}
\newcommand{\A}{\mathcal{A}}

\newcommand{\M}{\mathcal{M}}

\begin{document}

\bibliographystyle{plain}
\title[]{On the computational complexity of algebraic numbers: the Hartmanis--Stearns problem revisited}
\author{Boris Adamczewski}
\address{
CNRS, Universit\'e de Lyon, Universit\'e Lyon 1\\
Institut Camille Jordan  \\
43 boulevard du 11 novembre 1918 \\
69622 Villeurbanne Cedex, France}
\email{Boris.Adamczewski@math.cnrs.fr}

\author{Julien Cassaigne}

\address{
CNRS, Aix-Marseille Universit\'e \\
Institut de Math\'ematiques de Marseille  \\
case 907, 163 avenue de luminy \\
13288 Marseille Cedex 9, France  
}

\email{Julien.Cassaigne@math.cnrs.fr}

\author{Marion Le Gonidec}

\address{Universit\'e de la R\'eunion\\
Laboratoire d'Informatique et de Math\'ematiques \\
Parc technologique universitaire, 2 rue Joseph Wetzell\\
97490 Sainte-Clotilde, France
}

\email{marion.le-gonidec@univ-reunion.fr}

\thanks{This project has received funding from the European Research Council (ERC) under the European Union's Horizon 2020 research and innovation programme 
under the Grant Agreement No 648132.}
%%%%%%%%%%%%%%%%%%%%%%%%%%%%%%%%%%%%%%
\begin{abstract}  
We consider the complexity of integer base expansions of algebraic irrational numbers from a computational 
point of view.  A major contribution in this area is that the base-$b$ expansion of algebraic irrational real numbers cannot be 
generated by finite automata. Our aim is to provide two natural generalizations of this theorem. 
Our main result is that the base-$b$ expansion of algebraic irrational real numbers cannot be 
generated by deterministic pushdown automata.  
Incidentally, this completely solves the Hartmanis--Stearns problem for the class of 
multistack machines. We also confirm an old claim of Cobham from 1968 proving that such real numbers cannot be generated by  tag machines with dilation factor 
larger than one. In order to stick with the modern terminology, 
we also show that  the latter generate the same class of real numbers than morphisms with exponential growth. 
\end{abstract}
%%%%%%%%%%%%%%%%%%%%%%%%%%%%%%%%%%%%%%

\maketitle
%\tableofcontents
%%%%%%%%%%%%%%%%%%%%%%%%%%%%%%%%%%%%%%

\section{Introduction}

An old source of frustration for mathematicians arises from the study of integer base expansions 
of classical constants like 
$$
\begin{array}{cccc}
&\sqrt 2 &= &1.414\,213\, 562\, 373\, 095\, 048\, 801\, 688\, 724\, 209\, 698\, 078\, 569 \cdots \\ 
\text{ or} & &&\\

&  \pi& = &3.141\, 592\, 653\, 589\, 793\, 238\, 462\, 643\, 383\, 279\, 502\,884\, 197  \cdots 
\end{array}
$$
While these numbers admit very simple geometric descriptions, a close look at their digital expansion suggest highly 
complex phenomena. 
Over the years, different ways have been envisaged to formalize this old problem. This reoccurring 
theme appeared  in particular in three fundamental papers using: 
the language of probability after \'E. Borel \cite{Bo09}, the language of dynamical systems 
after Morse and Hedlund \cite{HM38}, and the language of Turing machines after Hartmanis and 
Stearns \cite{HS65}. Each of these points of view leads to a different assortment  of challenging conjectures.  
 As its title suggests, the present paper will focus on the latter approach. 
 It is addressed to researchers interested both in Number Theory and Theoretical Computer Science. 
 In this respect, we took care to make the paper as self-contained as possible and hopefully readable 
 by members from these different communities. 

\medskip

After the seminal work of Turing \cite{Turing}, real numbers can be rudely divided into two classes. 
On one side we find computable real numbers, those whose base-$b$ expansion can be produced 
by a Turing machine, while on the other side lie uncomputable real numbers which will belong for ever  
beyond the ability of  computers.  Note that, though most real numbers belong to the second class, 
classical mathematical constants are usually computable. This is in particular the case of any algebraic number.  
However, among computable numbers, 
some are quite easy to compute while others seem to have an inherent complexity that make 
them difficult to compute. In 1965, Hartmanis and Stearns \cite{HS65} investigated the fundamental 
question of how hard a real number may be to compute, introducing the now classical time complexity classes.   
The notion of time complexity takes into account the number $T(n)$ of operations needed by a multitape deterministic 
Turing machine to produce the first $n$ digits of the expansion. In this regard, a real number 
is considered all the more simple as its base-$b$ expansion can be produced very fast by a Turing machine.  
At the end of their paper, Hartmanis and Stearns suggested the following problem.

\medskip

\noindent {\bf\itshape Problem {\rm\bf HS}.} --- 
\emph{ Do there exist  irrational algebraic numbers for which the first $n$ binary digits can be computed in $O(n)$ 
operation by a multitape deterministic Turing machine?
}

\medskip

Let us  briefly recall why Problem HS is still open and likely uneasy to solve. 
On the one hand, all known approaches to compute efficiently the base-$b$ expansion 
of algebraic irrational numbers intimately relate on the cost of the multiplication $M(n)$ of two 
$n$-digits numbers (see for instance \cite{BB}).   This operation is computable in quasilinear 
time\footnote{This means computable in $O(n\log^{1+\varepsilon}n)$ operations for some $\varepsilon$.} 
but to determine whether one may have $M(n)=O(n)$ or not remains a famous open problem in this area.  
On the other hand, a negative answer to Problem HS\footnote{As observed in \cite{Co68}, this may 
be the less surprising issue.}  
would contain a powerful transcendental statement, 
a very special instance of which is  the transcendence of the following three simple irrational 
real-time 
%\footnote{Real-time computation 
%usually means that a new digit is produced at least 
%every $k$ computational steps for some fixed $k$. 
%In particular, this implies that $T(n)=O(n)$, but the converse does not hold.} 
computable numbers: 
$$
\sum_{n= 1}^{\infty} \frac{1}{2^{n!}} \;,\;  \sum_{n= 1}^{\infty} \frac{1}{2^{n^2}}\; \mbox{ and } 
 \sum_{n= 1}^{\infty} \frac{1}{2^{n^3}} \, \cdot
$$
Of course, for the first one, Liouville's inequality easily does the job. But the transcendance of the 
second number only dates back to 1996 \cite{Be97,DNNS} and its  
 proof requires  the deep work of 
Nesterenko  about algebraic independence of values of Eisenstein's series \cite{Ne96}. 
Finally, the transcendence of the third number remains unknown.

\medskip 

In 1968, Cobham \cite{Co68} (see also \cite{Co68c,Co68b}) was the first to consider the restriction of the Hartmanis-Stearns 
problem to some classes of Turing machines.   
The model of computation he originally investigated  is the so-called {\it Tag machine}.  
Though this model has some historical interest, this terminology 
is not much used today by the computer science community. However, their outputs precisely correspond to the class of 
\emph{morphic sequences}, a well-known object of interest for both mathematicians and computer scientists. 
They are especially used  in combinatorics on words and symbolic dynamics (see for instance \cite{AS,Py,Que}).  
In the sequel, we will present our result with this modern terminology but, for the interested reader, 
we will still describe the connection with tag machines in Section \ref{sec: tag}.  
In his paper, Cobham stated two main theorems without 
proof and only gave some hints that these statements should be deduced from a general transcendence method based on 
some functional equations, now known as Mahler's method. 
%\footnote{Note   
%that Cobham was apparently not aware about Mahler's contribution but only on works of Siegel about 
%$E$-functions thanks to Gelfond's book. See the interesting comments at the end of \cite{Co68b}.}.   
%In this direction, a great deal of works has been done by Loxton and van der Poorten \cite{LvdP1,LvdP2} 
%and Nishioka \cite{Ni90} among others, but it still remains a challenging problem to complete the proofs originally envisaged by Cobham. 
His first claim was finally confirmed by the first author and Bugeaud \cite{AB07}, but  using a totally different 
approach based on a $p$-adic version of the subspace Theorem (see \cite{AB07,ABL})\footnote{Very recently, some advances in Mahler's method \cite{PPH,AF} 
allow to complete the proof originally envisaged by Cobham.}. 

\medskip

\noindent {\bf\itshape Theorem {\rm\bf AB} (Cobham's first claim).} --- 
\emph{ The base-$b$ expansion of an algebraic irrational number cannot be generated by a uniform morphism or, 
equivalently, by a finite automaton.  
}

\medskip

\begin{rem} Theorem AB actually refers to two conceptually quite 
different models of computation: uniform morphisms and finite automata. 
There are two natural ways a multitape deterministic 
%\footnote{Contrary to their use in other classical contexts, 
%machines are always assumed to be deterministic when computing real numbers. This is due to the fact that they have to produce outputs. Thus all 
%machines under consideration in this paper will be deterministic.} 
Turing machine can be used to define computable numbers. 
First, it can be considered as {\it enumerator}, which means that the machine will produce one by one all the digits, 
separated by a special symbol, on its output tape.  Problem HS originally referred to the model of enumerators. 
The other way, referred to as {\it Turing transducer}, consists in feeding to the machine some input representing a positive 
integer $n$ and asking that the machine compute the $n$-th digit on its output tape.  
In Theorem AB, uniform morphisms (or originally uniform tag machines) are 
enumerators while finite automata are used as transducers\footnote{Note that, used as enumerators, finite automata can only produce 
eventually periodic sequences of digits and thus rational numbers. In contrast, used as transducers, finite automata output the interesting 
class of automatic sequences.}. The fact that these two models are equivalent is due to Cobham 
\cite{Co72}. 
\end{rem}

Theorem AB is the main contribution up to date toward a negative solution to Problem HS. In this paper, we show  
that the approach developed in \cite{AB07,ABL} leads to two interesting generalizations 
of this result.  Our first generalization is concerned with enumerators. In this direction, we 
confirm the second claim of Cobham.  
 
 \begin{thm}[Cobham's second claim]\label{thm:2bis}
The base-$b$ expansion of an algebraic irrational number cannot be generated by a morphism with exponential growth. 
 \end{thm}
 
%The definition of a real number generated by a morphism will be given in Section \ref{sec: morphism}.  
We stress that, stated as this, Theorem \ref{thm:2bis} also appeared in the Th\`ese de Doctorat of Julien Albert \cite{Alb}.    
In order to provide a self-contained proof of Cobham's second claim,  which was originally formulated in terms of tag machines, 
we will complete and reprove with permission some content of \cite{Alb} in Section \ref{sec:ja}.  
The fact that Theorem  \ref{thm:2bis} is well equivalent to 
Cobham's second claim will be proved in Section \ref{sec: tag}. \\

Our second and main generalization of Theorem AB is concerned with transducers. 
We consider a classical computation model called the {\it deterministic pushdown automaton}. 
It is of great importance on the one hand for theoretical aspects because of Chomsky's hierarchy \cite{Ch56} in formal language theory   
and on the other for practical applications, especially in parsing (see \cite{HMU}).   Roughly, such a device is a finite automaton with 
in addition a possibly infinite memory organized as a stack. 
Our main result is the following. 

\begin{thm}\label{thm:1bis}
  The base-$b$ expansion of an algebraic irrational number cannot be generated by a deterministic pushdown automaton.
\end{thm}

In Section  \ref{subsec:mult}, we use Theorem \ref{thm:1bis} to revisit the Hartmanis--Stearns problem as follows:  
instead of some time constraint,  we put some restriction based 
on the way the memory may be stored by Turing machines. 
This leads us to consider a classical computation model called \emph{multistack machine}. 
It corresponds to a version of the deterministic Turing machine where the memory is simply organized by stacks. 
It is as general as the Turing machine if one allows two or more stacks. 
Furthermore the one-stack machine turns out to be equivalent to the deterministic pushdown automaton, while a zero-stack machine is 
just the finite automaton of Theorem AB, that is a machine with a strictly finite memory only stored in the finite state control.     
Incidentally, Theorems AB and \ref{thm:1bis} turn out to completely solve the Hartmanis--Stearns problem for  multistack machines. 
Our approach  also provides a method to prove the transcendence of some real numbers generated by linearly bounded Turing 
machines (see the example in Section \ref{sec:beyond}).

\medskip

This paper is organized as follows. 
 Definitions related to finite automata,  morphisms, and pushdown automata are given in Section \ref{sec:2bis}.  
The useful combinatorial transcendence criterion of \cite{ABL}, on which our results are based, is recalled in Section \ref{sec:ABL}.   
Section \ref{sec:proof2bis} is devoted to the proofs of our main results: Theorems \ref{thm:2bis} and \ref{thm:1bis}. 
In connection with these results, two models of computation are discussed in Section \ref{sec:ts}: the tag machine and the multistack machine. 
Finally, Section \ref{sec:end} is devoted to concluding remarks regarding factor complexity, some quantitative aspects of this method, 
and continued fractions.

%%%%%%%%%%%%%%%%%%%%%%%%%%%%%%%%%%%%
\section{Finite automata, morphic sequences and pushdown automata}\label{sec:2bis}

In this section, we give definitions for a real number to be generated by a 
finite automaton, a morphism, and a deterministic pushdown automaton. 
This provides a precise meaning to Theorems AB, \ref{thm:2bis}, and \ref{thm:1bis}. 

\medskip

All along this paper, we will use the following notations. 
An alphabet $A$ is a finite set of symbols, also called letters. A finite word over $A$ is a 
finite sequence of letters in $A$  or equivalently an element of $A^*$, the free monoid  generated by $A$.  
The length of a finite word $W$, that is the number 
of symbols composing $W$, is denoted by $\vert W\vert$. 
We will denote by $\epsilon$ the empty word, that is the unique word of length $0$, and $A^+$ the set of finite words of positive length over $A$. 
If $a$ is a letter and $W$ a finite word,  
then $\vert W\vert_a$ stands for the number of occurrences of the letter $a$ in $W$. 
Let $k\ge 2$ be a natural number.  
We let $\Sigma_k$ denote the alphabet  $\left\{0,1,\ldots,k-1\right\}$. Given a positive integer $n$, we set 
$\langle n\rangle_k :=w_r w_{r-1}\cdots w_1 w_0$ for the base-$k$ expansion of $n$, which means that $n=\sum_{i=0}^r w_i k^{i}$ 
with $w_i\in \Sigma_k$ and $w_r\not= 0$.  Note that by convention $\langle 0\rangle_k:= \epsilon$. 
Conversely, if $w := w_1 \cdots w_r$  is a finite word over the alphabet $\Sigma_k$, we set $[w]_k :=\sum_{i=0}^r w_{r-i} k^{i}$.  
The usual notations $\{x\}$, $\lfloor x \rfloor$, and 
$\lceil x\rceil$  respectively stand for the fractional part, the floor, and the ceil  of the real number $x$.

%%%%%%%%%%%%%%%%%%%%%%%%%%%%%%%
\subsection{Finite automata and automatic sequences} 

A sequence ${\bf a}:= (a_n)_{n\geq 0}$ with values in a finite set 
is \emph{$k$-automatic} if it can be generated by a finite automaton used as a transducer. 
This means that 
there exists a finite-state machine 
(a deterministic finite automaton with output) that takes as input the base-$k$ expansion of 
$n$ and produces as output the symbol $a_n$.  

Let us give now a formal definition of a $k$-automatic sequence. 
Let $k\geq 2$ be a natural number. A $k$-automaton is defined as a $6$-tuple 
$\displaystyle{\mathcal A} = \left(Q,\Sigma_k,\delta,q_0,\Delta,\tau\right) ,
$
where:
\begin{itemize}
\item  $Q$ is a finite set of states, 
\item $\delta:Q\times \Sigma_k\rightarrow Q$ is the transition function,
\item $q_0$ is 
the initial state,
\item $\Delta$ is the output alphabet,
\item and $\tau : Q\rightarrow 
\Delta$ is the output function.
\end{itemize}
 Given a state $q$ in $Q$ and a finite 
word $w=w_1 w_2 \cdots w_n$ on the alphabet $\Sigma_k$, 
we define $\delta(q,w)$ recursively by 
$\delta(q,w)=\delta(\delta(q,w_1w_2\cdots w_{n-1}),w_n)$.  

\begin{defn}
Let $\displaystyle{\mathcal A} = \left(Q,\Sigma_k,\delta,q_0,\Delta,\tau\right)$ be a $k$-automaton. 
The output sequence produced by $\mathcal A$ is the sequence $(\tau(\delta(q_0,\langle n\rangle_k)))_{n\geq 1}$. 
Such a sequence is called a {\it $k$-automatic sequence}. A sequence or an infinite word is said to be automatic if it is $k$-automatic for some integer $k\geq 2$. 
\end{defn}

\begin{exam} By a classical result of Lagrange, it is known that every non-negative integer can be written as the sum of four perfect squares.  
It is optimal in the sense that some natural numbers cannot be written as the sum of only three square numbers.  More precisely, Legendre proved that  
$$
\exists\, a,b,c \mid n=a^2+b^2+c^2 \iff \not\exists\, i,j \mid n = 4^{i}(8j+7)\,,
$$ 
where $n,a,b,c,i,j$ are nonnegative integers. 
As a consequence, the binary sequence ${\bf s}:=(s_n)_{n\geq 0}$ defined by $s_n=1$ if $n$ can be written as 
the sum of three squares and $s_n=0$ otherwise is $2$-automatic.  
\begin{figure}[h]
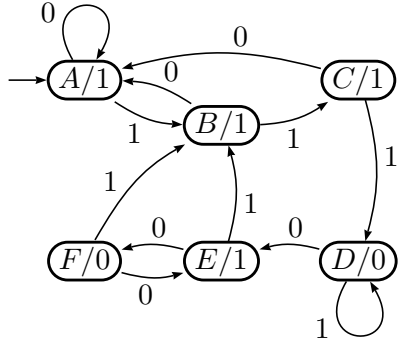

\centering
\VCDraw{%
        \begin{VCPicture}{(0,2)(6,-3)}
% states
 \StateVar[A/1]{(0,0)}{a}  \StateVar[B/1]{(3,-1)}{b} \StateVar[C/1]{(6,0)}{c}
  \StateVar[F/0]{(0,-4)}{f}  \StateVar[E/1]{(3,-4)}{e} \StateVar[D/0]{(6,-4)}{d}
% initial--final
\Initial[w]{a}
% transitions 
\LoopN{a}{0}
\LoopS{d}{1}
\ArcR{a}{b}{1}
\ArcR{b}{a}{0}

\ArcR{e}{b}{1}
\ArcR{e}{f}{0}
\ArcR{f}{e}{0}
\ArcL{f}{b}{1}
\ArcR{c}{a}{0}
\ArcL{c}{d}{1} 
\ArcR{b}{c}{1}
\ArcR{d}{e}{0}

\end{VCPicture}
}\caption{A $2$-automaton generating the sequence ${\bf s}$.}
  \label{figure0}
\end{figure}
\end{exam}

 \begin{defn}
A real number $\xi$ can be generated by a deterministic $k$-automaton $\mathcal A$ if, 
for some integer $b\geq 2$, one has $\langle\{\xi\}\rangle_b=0.a_1a_2\cdots$, where $(a_n)_{n\geq 1}$ 
corresponds to the output sequence produced by $\mathcal A$.
\end{defn}

Theorem AB thus implies that the binary number 
$$
\xi_0:= 1. 111 \, 111\, 011\, 111\, 110\, 111\, 111\, 111\, 111\, 011\, 011\, 111\, 110\cdots
$$
generated by the $2$-automaton of Figure \ref{figure0} is transcendental. 
%%%%%%%%%%%%%%%%%%%%%%%%%%%%%%%%%%
\subsection{Morphic sequences}\label{sec: morphism}

As already mentioned in the introduction, Cobham suggested in 1968 to restrict Problem HS to a special 
class of real-time Turing machines called tag machines. Though these machines have an historical interest, this terminology 
is not much used today by the computer science community. However, their outputs precisely correspond to class of 
morphic sequences that we define below. In contrast, the latter are well-known objects of interest for both mathematicians and computer scientists. 
For the interested reader, we will still describe tag machines in Section \ref{sec: tag}.

\medskip

We recall here some basic definitions. 
Let $A$ be a finite  alphabet. 
A map from $A$ to $A^*$ naturally extends to a map from $A^*$ into itself called (endo)\emph{morphism}.  
Given two alphabets $A$ and $B$, a map from $A$ to $B$ naturally extends to a map from $A^*$ into $B^*$ 
called a \emph{coding} or \emph{letter-to-letter morphism}.  
A morphism $\sigma$ over $A$ is said to be $k$-\emph{uniform} if 
$\vert \sigma(a)\vert =k$ for every letter $a$ in $A$, and just 
\emph{uniform} if it is $k$-uniform for some $k$.  
A useful object associated with a morphism $\sigma$ is the so-called {\it incidence matrix} of $\sigma$,  
denoted by $M_{\sigma}$.  We first need to choose an ordering of the elements of $A$, say 
$A=\{a_1,a_2,\ldots,a_d\}$, and then $M_{\sigma}$ 
is defined by 
$$
\forall i,j\in \{1,\ldots,d\},\;\; \left(M_{\sigma}\right)_{i,j} := \vert \sigma(a_j)\vert_{a_i} \, .
$$ 
The choice of the ordering has no importance. 
A morphism $\sigma$ over $A$ is said to be \emph{prolongable} 
on $a$ if $\sigma(a)=aW$ for some word $W$ and if the length of the word 
$\sigma^n(a)$ tends to infinity with $n$. Then the word 
$$
\sigma^{\omega}(a) := \lim_{n\to\infty} \sigma^n(a) = aW\sigma(W)\sigma^{2}(W)\cdots
$$
is the unique fixed point of $\sigma$ that begins with $a$. 
An infinite word obtained by iterating a prolongeable morphism $\sigma$ is said to be {\it purely morphic}. 
The image of a purely morphic word under a coding is a {\it morphic word}. 
Thus to define a morphic word ${\bf a}$ one needs a $5$-tuple ${\mathcal T} = (A,\sigma,a,B,\varphi)$ 
such that ${\bf a} = \varphi(\sigma^{\omega}(a))$, where:

\begin{itemize}

\medskip

  \item
 $A$ is a  finite set of symbols called the \emph{internal alphabet}. 
 
  \medskip
 
  \item 
$a$ is a an element of $A$ called the \emph{starting symbol}.
 
 \medskip
 
  \item 
$\sigma$ is a \emph{morphism} of $A^*$ prolongable on $a$.

\medskip

  \item 
 $B$ is a  finite set of symbols called the \emph{external alphabet}. 
 
  \medskip
 
  \item 
$\varphi$ is a letter-to-letter morphism from $A$ to $B$.

 \end{itemize}
 When $\sigma$ is uniform (resp.\ $k$-uniform, ), the sequence ${\bf a}$ is said to be generated by a uniform (resp.\ $k$-uniform) morphism. 

\begin{defn} A morphism $\sigma$ is said to have \emph{exponential growth} 
if the spectral radius of the matrix $M_{\sigma}$ is larger than one. 
When $\sigma$ has exponential growth and all letters of $A$ appear in $\sigma^{\omega}(a)$, the sequence ${\bf a}= \varphi(\sigma^{\omega}(a))$ 
is said to be generated by a 
morphism with exponential growth. 
\end{defn}

\begin{defn} A real number $\xi$ can be generated by a morphism with exponential growth if, 
for some integer $b\geq 2$, one has $\langle\{\xi\}\rangle_b=0.a_1a_2\cdots$, where ${\bf a}:=(a_n)_{n\geq 1}$ 
can be generated by a morphism with exponential growth. 
\end{defn}

Theorem \ref{thm:2bis} thus implies the transcendence of  the ternary number 
$$\xi_1 := 0.021\, 201\, 220\,210\, 122\, 202\, 120\, 120\, 210\, 122\, 220\, 212\,122\cdots$$
whose expansion is the sequence ${\bf a} = \varphi_1(\sigma_1^{\omega}(a))$, where $\sigma_1(a)= acb$, $\sigma_1(b)=abc$, $\sigma_1(c)=c$,  
$\varphi_1(a)=0$, $\varphi_1(b)=1$, and $\varphi_1(c)=2$. One can check that $\sigma_1$ has exponential growth for the spectral radius of $M_{\sigma_1}=2$. 
In contrast, Theorem \ref{thm:2bis} does not apply to prove the transcendence of binary number $\displaystyle\sum_{n=1}^{\infty}\frac{1}{2^{n^2}}$. 
Indeed, though this number can be generated by a morphism, the latter has non-exponential growth. Indeed the characteristic sequence of squares   
can be obtained as $\varphi_2(\sigma_2^{\omega}(a))$ where $\sigma_2(a)= ab$, $\sigma_2(b)=ccb$, $\sigma_2(c)=c$,  
$\varphi_2(a)=\varphi_2(c)=0$, and $\varphi_2(b)=1$. One can check easily that the spectral radius of $M_{\sigma_2}=1$.

\begin{rem}
Following Cobham \cite{Co68}, there is no loss of generality to assume than the internal morphism 
$\sigma$ is a non-erasing morphism, which means that no letter is mapped to the empty word. Indeed, if an infinite word can be generated by  
an erasing morphism , then there also 
exists a non-erasing morphism that can generate it. From now on, we will 
thus only consider non-erasing morphisms. 
\end{rem}

\medskip

It is worth mentioning that the class of sequences generated by uniform morphisms is 
especially relevant  because of the following result of Cobham \cite{Co72}. 

\medskip

\noindent {\bf\itshape Proposition {\rm\bf C}.} --- 
\emph{ A sequence is $k$-automatic if and only if 
it can be generated by some $k$-uniform morphism. 
}

\medskip
 Furthermore the proof of Proposition C is completely constructive and provides a simple way to go from   
$k$-uniform morphisms to $k$-automata and {\it vice versa}. This general feature is examplified below. 
For a complete treatment see \cite{Co72} or Chapter 6 of \cite{AS}. 

\begin{exam}
The Thue--Morse sequence\index{Thue--Morse sequence} ${\bf t}:=(t_n)_{n\geq 0}$ 
is probably the most famous example among automatic sequences. It is defined as follows: 
$t_n=0$ if the sum of the binary digits of $n$ is even, and $t_n=1$ otherwise.  
It can be generated by the following 
finite $2$-automaton:    
$
{\mathcal A}=\left ( \{q_0, q_1\}, \{0, 1\}, \delta, q_0, \{0, 1\}, \tau \right)$, 
where
$
\delta(q_0, 0) = \delta (q_1, 1) = q_0$, $\delta(q_0, 1) = \delta (q_1, 0) = q_1$,  
 $\tau (q_0) = 0$  and  $\tau (q_1) = 1$. 
\begin{figure}[htbp]
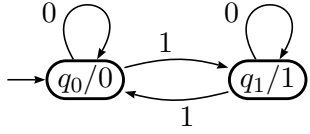

\centering
\VCDraw{%
        \begin{VCPicture}{(0,-1)(4,2)}
% states
 \StateVar[q_0/0]{(0,0)}{a}  \StateVar[q_1/1]{(4,0)}{b}
% initial--final
\Initial[w]{a}
% transitions 
\LoopN{a}{0}
\LoopN{b}{0}
\ArcL{a}{b}{1}
\ArcL{b}{a}{1}
\end{VCPicture}
}
\caption{A $2$-automaton generating Thue--Morse sequence.}
  \label{AB:figure:thue}
\end{figure}
The Thue--Morse sequence can be as well generated by a $2$-uniform morphism for one has  
${\bf t}= \varphi(\sigma^{\omega}(q_0))$, where 
$\sigma(q_0)=q_0q_1$, $\sigma(q_1)=q_1q_0$,  $\varphi(q_0)=0$, $\varphi(q_1)=1$. 
\end{exam}

%%%%%%%%%%%%%%%%%%%%%%%%%%%%%%%%
\subsection{Pushdown automata}\label{sec:PDA} 

A pushdown automaton is a classical device, but most often used in formal language theory as an acceptor, 
that is a machine that can accept or reject finite words and langages, namely context-free languages (see for instance \cite{Au,ABB,HMU}).    
Our point of view here is slightly different  
for we will use the pushdown automaton as a transducer, that is a machine that associates a symbol 
with every finite word on a given input alphabet.  

Formally, a $k$-{pushdown automaton} is a  complete deterministic pushdown automaton with output, or DPAO for short.  
It is defined as a 7-tuple $\M =(Q,\Sigma_k,\Gamma,\delta,q_0,\Delta,\tau)$ where: 
\begin{itemize}

\medskip

  \item
 $Q$ is a  finite set of \emph{states},
 
  \medskip
 
  \item 
$\Sigma_k:=\{0,1,\ldots,k-1\}$ is the finite set of \emph{input symbols}, 
 \medskip
 
  \item 
$\Gamma$ is the finite set of \emph{stack symbols}. A special symbol $\#$ is used to mark the bottom of the stack.
\\  
For convenience, we identify $\#$ with the empty word of $\Gamma^*$.  

\medskip

  \item 
$\delta: E\subset  Q\times \left(\Gamma\cup\{\#\}\right)  \times \left(\Sigma_k\cup\{\varepsilon\}\right) \to Q\times \Gamma^* $ 
 is the \emph{transition  function},
 
  \medskip
 
  \item 
$q_0\in Q$ is the {\it initial state} and $(q_0,\#)$ is the {\it initial (internal) configuration}, 

\medskip
 
  \item 
$\Delta$ is the finite set of {\it output symbols},

 \medskip
 
  \item 
$\tau:Q\times \Gamma\cup\{\#\} \to \Delta$ is the {\it output function}. 
 \end{itemize}

\medskip
 
 Furthermore, the transition function satisfies the following conditions.
 
 \medskip
 
  \begin{itemize}
  \item  {\it Determinism assumption}:
 if $(q,a,\epsilon)$ belongs to $E$ for some $(q,a)\in Q\times \left(\Gamma\cup\{\#\}\right)$, then for every $i\in\Sigma_k$, 
 $(q,a,i)\not\in E$.
 
 \medskip
 
  \item {\it Completeness assumption}:  If $(q,a,\epsilon)$ does not belong to $E$ for some $(q,a)\in Q\times \left(\Gamma\cup\{\#\}\right)$, 
  then $ \{q\} \times \{a\} \times \Sigma_k \subset E$.

\end{itemize}

\begin{rem}\label{rem:det}
{\rm  Notice that $\delta$ being a function is also a part of the determinism assumption. 
In a nondeterministic $k$-pushdown automata, $\delta$ would be only define as a subset of 
$Q\times  \left(\Gamma\cup\{\#\}\right) \times \left(\Sigma_k\cup\{\epsilon\}\right)\times Q\times \Gamma^* $.}
\end{rem}

We want now to make sense to the computation $\tau(\delta(q_0,\#,W))$ 
 for any input word $W$ in $\Sigma_k^*$.   
First the transition function $\delta$ of a $k$-pushdown automaton  
can naturally be extended to  a subset 
of $Q\times \Gamma^*\times \left(\Sigma_k\cup\{\epsilon\}\right)$ by 
setting
$$
\forall S=s_1\cdots s_j \in \Gamma^*, \vert S\vert \geq 2 \,,\;\; \delta(q,S,a) = (q',s_1\cdots s_{j-1}X)\, ,
$$
when $\delta(q,s_j,a)=(q',X)$.
 After reading the symbol $a$, the pushdown automaton could have reached a configuration $(q,S)$ from which 
$\epsilon$-moves are possible. In such a case, one asks the pushdown automaton to perform all possible $\epsilon$-moves 
before reading the next input symbol.  We stress that this appears to be a classical convention (see the discussion in \cite{ABB})  
which can be formalized as follows. 
We define  the function $\overline{\delta}$ by: 
 $$
 \overline{\delta}(q,S,a)= \left\{
   \begin{array}{cl}
   \delta(q,S,a) & \mbox{if }(\delta(q,S,a),\epsilon)\notin E  \\
   \overline{\delta}(\delta(q,S,a),\epsilon) & \mbox{if }(\delta(q,S,a),\epsilon)\in E   
   \end{array}
 \right.$$
Then  $\overline \delta$ can be extended to a subset of $Q\times \Gamma ^*\times \Sigma_k^*$ 
by setting 
 $$
 \overline{\delta}(q,S,w_1\cdots w_r)=\overline{\delta}\left(\overline{\delta}\left(q,S,w_1\cdots w_{r-1}\right),w_r  \right)\, .
 $$
 This means in particular that $\mathcal M$ scans its inputs from left to right. From now on, we will not distinguish $\delta$ from its extension 
 $\overline{\delta}$. 
 We also extend the output function $\tau$ to a subset of $Q\times \Gamma^*$ by simply setting 
 $\tau(q,s_1s_2\cdots s_j)=\tau(q,s_j)$.

 \begin{defn} 
 Let  $\mathcal M =(Q,\Sigma_k, \Gamma,\delta,q_0,\Delta,\tau)$ be a $k$-pushdown automaton.   
 The sequence $(\tau(\delta(q_0,\#,\langle n\rangle_k)))_{n\geq 1}$ is called the output sequence produced by
 $\mathcal M$. 
 \end{defn}
 
This class of sequences are discussed in \cite{CLG}. They form a subclass of context-free sequences (see~\cite{LG12,Mo08}).

 \begin{exam}\label{ex:pushdown} Usually a deterministic pushdown automaton is represented as a finite graph whose vertices are labelled by the elements of $Q$ and whose edges 
 are labelled by transitions as follows: $\delta(q,S,a)=q'W$ is represented by the edge $q\xrightarrow[]{(a,S|W)} q'$. An example of such internal representation is given 
 by the $2$-pushdown automaton $\mathcal A$ in Figure \ref{fig:AP}.  
 It outputs the binary sequence  $${\bf a}:= 1110111001101000011111101110100000010110\cdots$$ 
whose $n$-th binary digit   
is $1$ if the difference between the number of occurrences of the digits $0$ and $1$ in the binary expansion of $n$ is at 
most $1$, and is $0$ otherwise. 

 \begin{figure}[htbp]
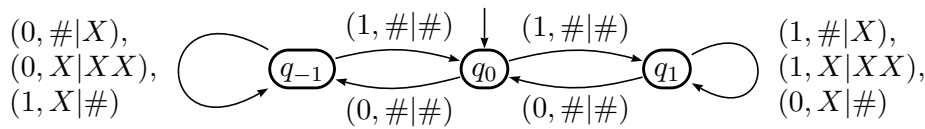

\centering
\VCDraw{%

 \begin{VCPicture}{(-5,-1)(5,1.5)}
% states
% \StateVar[A]{(0,2)}{s} 
 \StateVar[q_1]{(4,0)}{p} \StateVar[q_0]{(0,0)}{q} \StateVar[q_{-1}]{(-4,0)}{r}
% initial--final
\Initial[n]{q}
% transitions 

%\EdgeR{s}{p}{(1,\#|\#)}
%\LoopN{s}{(0,\#|\#)}
\LoopE[.5]{p}{
  \begin{array}{l}
  (1,\#|X),\\(1,X|XX),\\(0,X|\#)  
  \end{array}
}

\ArcL[.5]{p}{q}{(0,\#|\#)}
\ArcL[.5]{q}{p}{(1,\#|\#)}
\ArcL[.5]{r}{q}{(1,\#|\#)}
\ArcL[.5]{q}{r}{(0,\#|\#)}

\LoopW[.5]{r}{ \begin{array}{l}(0,\#|X),\\(0,X|XX),\\(1,X|\#) \end{array}}
\end{VCPicture}
}
  \caption{A $2$-pushdown automaton producing the binary expansion of $\xi_2$}
   \label{fig:AP}

\end{figure}
This automaton works as follows. Being on state $q_0$ means that the part of the input word that has been already read contains as many $1$'s as $0$'s. 
On the other hand, being on state $q_1$ means that the part of the input word that has been already read contains more $1$'s than $0$'s, while 
being on state $q_{-1}$ means that it contains more $0$'s than $1$'s. Furthermore, in any of these two states, the difference between the number of $0$'s and $1$'s (in absolute value) 
is one more than the number of $X$'s in the stack. Thus, the difference between the number of occurrences of the symbols $1$ and $0$ in the input word is at most $1$ if, and only if, 
the reading ends with an empty stack (regardless to the ending state). By definition of the output function, we see that $\mathcal A$ well produces the infinite word ${\bf a}$. 

Note that, formally, this pushdown automaton should be defined as:  $\mathcal A:=(\{q_0,q_1,q_{-1}\}, \Sigma_2,\{X\},\delta,q_0, \{0,1\}, \tau)$, 
where  the transition function $\delta$ is defined by 
$\delta(q_0,\#,1)=(q_1,\#)$, $\delta(q_0,\#,0)=(q_{-1},\#)$, $\delta(q_1,\#,0)=(q_0,\#)$, $\delta(q_1,\#,1)=(q_1,X)$, 
$\delta(q_1,X,0)=(q_1,\#)$, $\delta(q_1,X,1)=(q_1,XX)$, 
 $\delta(q_{-1},\#,0)=(q_{-1},X)$, $\delta(q_{-1},\#,1)=(q_0,\#)$, $\delta(q_{-1},X,0)=(q_{-1},XX)$, $\delta(q_{-1},X,1)=(q_{-1},\#)$, 
and where the output function $\tau$ is defined by 
$\tau(q_0,\#)=\tau(q_1,\#)=\tau(q_{-1},\#)=1$, and $\tau(q_0,X)=\tau(q_1,X)=\tau(q_{-1},X)=0$.   
\end{exam}
 
 \begin{defn}
A real number $\xi$ can be generated by a deterministic $k$-pushdown automaton  $\mathcal M$ if, 
for some integer $b\geq 2$, one has $\langle\{\xi\}\rangle_b=0.a_1a_2\cdots$, where $(a_n)_{n\geq 1}$ 
corresponds to the output sequence produced by $\mathcal M$.
\end{defn}

Theorem \ref{thm:1bis} thus implies the transcendence of  the binary number   
$$\xi_2 := 1.110\, 111\, 001\, 101\, 000\, 011\, 111\, 101\, 110\, 100\, 000\, 010\, 110\cdots$$ 
whose binary expansion is the infinite word ${\bf a}$ of Example \ref{ex:pushdown}.  
 
 \begin{rem}\label{subsec:eps}
 
\emph{About $\epsilon$-moves.}--- 
Since we only consider deterministic pushdown automata,  
 we can assume without loss of generality that all $\epsilon$-moves are decreasing (see for instance \cite{ABB}). 
 This means that a computation of the form $\delta(q,W,\epsilon)=(q',W')$, always implies that $\vert W'\vert < \vert W\vert$.  
 
 \medskip
 
\emph{About input words}---  
In our model of $k$-pushdown automaton, we choose to feed our machines only with the 
proper base-$k$ expansion of each nonnegative integer $n$. 
Instead, we could as well  imagine to ask that $\tau(\delta(q_0,\#,w))$ remains the same for all words $w\in\Sigma_k^*$ 
such that $[w]_k=n$, that is $\tau(\delta(q_0,\#,w))=\tau(\delta(q_0,\#,0^jw))$ for every natural number $j$.  
Such a change would not affect the class of output sequences produced by $k$-pushdown automata.  
The discussion is similar to the case of the $k$-automaton and we refer to \cite{AS} for more details. 

 Our second remark concerning inputs is more important. In our model, the $k$-pushdown automaton scans the 
 base-$k$ expansion of a positive integer $n$ starting from the most significant digit. This corresponds to the usual way humans 
 read numbers, that is from left to right. In the case of the $k$-automaton, this choice is of no consequence because both ways  
 of reading are known to be equivalent.  However, this is no longer true for $k$-pushdown machines as the class of deterministic context free languages is not closed under mirror 
 image.  

\medskip
 
\emph{About uniqueness}--- There always exist several different $k$-pushdown automata producing the same output. 
   In particular, it is  possible to choose one with a single state (see for instance \cite{Au}). 
   The $2$-automaton given in Figure \ref{fig:AP} is certainly not the smallest one with respect to 
   the number of states, but it  makes the process of computation more transparent and it only uses one ordinary stack symbol.  
\end{rem}

%%%%%%%%%%%%%%%%%%%%%%%%%%%%%%%%%%%%
\section{A combinatorial transcendence criterion}\label{sec:ABL}
In this section, we recall the fundamental relation between Diophantine approximation and repetitive patterns occurring in integer base expansions of real numbers.   

\medskip

Let $A$ be an alphabet and $W$ be a finite word over $A$. 
For any positive integer $k$, we write
$W^k$ for the word 
$$\underbrace{W\cdots W}_{\mbox{$k$ times}}$$ 
(the concatenation
of the word $W$ repeated $k$ times). 
More generally, for any positive real number
$x$,  $W^x$ denotes the word
$W^{\lfloor x \rfloor}W'$, where $W'$ is the prefix of
$W$ of length $\left\lceil \{x\}\vert W\vert\right\rceil$. 
The following natural measure of periodicity for infinite words was  
introduced in \cite{AB07a} (see also \cite{Ad,AC06}).

\begin{defn}\label{def:dio}{\rm The {\it Diophantine exponent} 
of an infinite word ${\bf a}$ is defined as 
the supremum of the real numbers $\rho$ for which there exist arbitrarily long prefixes of ${\bf a}$ 
that can be factorized as $UV^{\alpha}$, where $U$ and $V$ are two finite words 
($U$ possibly empty) and $\alpha$ is a real number such that  
$$
\frac{\vert UV^{\alpha}\vert}{\vert UV \vert} \geq \rho.
$$
The  Diophantine exponent of ${\bf a}$ 
is denoted by 
${\Dio}({\bf a})$.}
\end{defn}

 Of course, for any infinite word ${\bf a}$ one has the following relation
\begin{equation*}
1 \leq  \Dio({\bf a}) \leq +\infty.
\end{equation*}
Furthermore, $\Dio({\bf a}) = +\infty$ for an eventually periodic word ${\bf a}$, but the converse is not true. 
There is some interesting interplay between the Diophantine exponent and 
Diophantine approximation, which is actually reponsible for the name of the exponent.   
Let $\xi$ be a real number whose base-$b$ expansion is 
$0.a_1a_2\cdots$. Set ${\bf a} := a_1a_2\cdots$.  
Let us assume that the word ${\bf a}$ begins with a prefix of the form $UV^{\alpha}$.
Set $q = b^{\vert U\vert}(b^{\vert V\vert}-1)$. 
A simple computation shows that there exists an integer $p$ such that 
$$
\langle p/q\rangle_b = 0.UVVV\cdots.
$$  
Since $\xi$ and $p/q$ have the same first $\vert UV^{\alpha} \vert$ digits in their base-$b$ expansion,  
we obtain that 
$$
\left\vert \xi - \frac{p}{q} \right\vert < \frac{1}{b^{\vert UV^{\alpha}\vert}}
$$
and thus 
\begin{equation}\label{min}
\left\vert \xi - \frac{p}{q} \right\vert < \frac{1}{q^{\rho}},
\end{equation}
where $\rho = \vert UV^{\alpha}\vert/\vert UV\vert$. 

We do not claim here that $p/q$ is written in lowest terms.  Actually, it may well happen that the 
$\gcd$ of $p$ and $q$ is quite large but (\ref{min}) still holds in that case. 
By Definition~\ref{def:dio}, it follows that if $\Dio({\bf a})= \mu$, then for every $\rho < \mu$, there exists infinitely 
many rational numbers $p/q$ such that 
\begin{equation*}
\left\vert \xi - \frac{p}{q} \right\vert < \frac{1}{q^{\rho}} \cdot
\end{equation*} 
Note that when $\Dio({\bf a}) < 2$, such approximations look quite bad, for the existence of much better ones is ensured by 
the theory of continued fractions or by Dirichlet pigeonhole principle.  Quite surprisingly, the inequality $\Dio({\bf a}) > 1$ 
is already enough to conclude that $\xi$ is either rational or transcendental. This powerful combinatorial transcendence criterion,  
proved in \cite{ABL} and restated in Proposition ABL, emphasizes the relevance of the Diophantine exponent for our purpose. 

\medskip

\noindent {\bf\itshape Proposition {\rm\bf ABL}.} --- 
\emph{ Let $\xi$ be a real number with $\langle \{\xi\}\rangle_b :=  0.a_1a_2\cdots$. 
 Let us assume that $\Dio({\bf a})>1$ where ${\bf a}:= a_1a_2\cdots$. 
 Then $\xi$ is either rational or transcendental. 
}

\medskip

Proposition ABL is obtained as a consequence of the $p$-adic Subspace Theorem. 
It is the key tool for proving Theorem AB and it 
will be the key tool for proving Theorems \ref{thm:2bis} and \ref{thm:1bis} as well. In this section, we recall the fundamental relation between Diophantine approximation and repetitive patterns occurring in integer base expansions of real numbers.

%%%%%%%%%%%%%%%%%%%%%%%%%%%%%%%%%%%%%%%%%%
\section{Proof of Theorems~\ref{thm:2bis} and ~\ref{thm:1bis}}\label{sec:proof2bis}

In this section, we prove our two main results.

\subsection{Proof of Theorem~\ref{thm:2bis}}\label{sec:ja}
In order to prove Theorem~\ref{thm:2bis}, we first need the following definition.

\begin{defn}
{\rm Let $A$ be a finite set and $\sigma$ be a morphism of $A^*$. A letter $b\in A$ 
is said to have {\it maximal growth} if there exists a real number $C$ such that 
$$
\vert \sigma^n(c) \vert < C\vert \sigma^n(b)\vert  \,,
$$
for every letter $c\in A$ and every positive integer $n$. }
\end{defn}

\begin{lem}\label{lem:max}
Let $A$ be a finite set and $a\in A$. Let $\sigma$ be a morphism of $A^*$ prolongable on $a$ and such that all letters of $A$ appear in 
$\sigma^{\omega}(a)$. 
Let $\theta$ denote the spectral radius of $M_{\sigma}$. 
Then the letter $a$ has maximal growth.  

Furthermore, there exist a nonnegative integer $k$, and two positive real numbers $c_1$ and $c_2$ 
such that 
\begin{equation}\label{eq:theta}
c_1n^k \theta^n  < \vert \sigma^n(a)\vert  < c_2 n^k \theta^n \,. 
\end{equation}
\end{lem}

\begin{proof}
Let $c$ be a letter occurring in $\sigma^{\omega}(a)$. Then $c$ also occurs in $\sigma^r(a)$, for some positive integer $r$. Then 
$$\vert \sigma^n(c)\vert  \leq \vert \sigma^{n+r}(a)\vert 
=\vert \sigma^r(\sigma^n(a))\vert \leq \Vert M_{\sigma^r}\Vert_{\infty} \vert \sigma^n(a)\vert\, ,
$$ 
where $\Vert\cdot\Vert_{\infty}$ stands for the usual infinite norm.   
This shows that $a$ has maximal growth. 
Recall now that by a classical result of Salomaa and Soittola (see for instance Theorem 4.7.15 in \cite{CaNi}), 
there exist a nonnegative integer $k$, a real number $\beta\geq 1$,  
and two positive real numbers $c_1$ and $c_2$ such that 
\begin{equation}\label{eq:theta2}
c_1n^k \beta^n  < \vert \sigma^n(a)\vert  < c_2 n^k \beta^n \,,
\end{equation}
for every positive integer $n$. Since $a$ has maximal growth, a classical theorem on matrices due to 
Gelfand (see for instance \cite{CaNi}) implies that $\beta$ must be equal to $\theta$, the spectral radius of the incidence 
matrix of $\sigma$.  
\end{proof}

\begin{prop}\label{prop:diotag}
Let ${\bf a}$ denote an infinite sequence than can be generated by a morphism with exponential growth.  
Then $\Dio({\bf a})>1$. 
\end{prop}

\begin{proof}
Let ${\bf a}$ denote an infinite sequence than can be generated by a morphism with exponential growth.  
Then ${\bf a}=\varphi(\sigma^{\omega}(a))$ for some morphism $\sigma$ with exponential growth defined over a finite alphabet $A$, and some coding $\varphi$. 
Furthermore, we recall that the spectral radius $\theta$ of the incidence matrix $M_\sigma$ satisfies $\theta >1$. 
Set ${\bf u}:= \sigma^{\omega}(a)$.   
Since by definition $\sigma$ is prolongable on $a$ and all letters of $A$ appear in ${\bf u}$, 
Lemma \ref{lem:max} implies that $a$ has maximal growth and that there exist two positive real numbers $c_1$ and $c_2$ such that 
\begin{equation}\label{eq:c2}
c_1n^k \theta^n  < \vert \sigma^n(a)\vert  < c_2 n^k \theta^n,
\end{equation} 
for every positive integer $n$. 

\medskip

We now prove that there are infinitely many occurrences of letters with maximal growth in ${\bf u}$. Let us argue by contradiction. 
If there are only finitely many occurrences of letters with maximal growth, then there exists a positive integer $n_0$ such that 
${\bf u}=\sigma^{n_0}(a){\bf w}$ where ${\bf w}$ is an infinite word that contains no letter with maximal growth. Since $\theta>1$, there is an integer 
$m_0$ such that 
\begin{equation}\label{eq:c2'}
c_2/\theta^{m_0} < c_1/2 .
\end{equation}
Let us denote by $V_0$ the unique finite word such that $\sigma^{n_0+m_0}(a) = \sigma^{n_0}(a)V_0$. 
Then for every positive integer $n$ we get that 
$$
\vert\sigma^{n+n_0+m_0}(a) \vert = \vert \sigma^{n+n_0}(a) \vert +\vert  \sigma^n(V_0)\vert 
\leq c_2(n+n_0)^k \theta^{n+n_0} + \vert  \sigma^n(V_0)\vert \, .
$$
Given $\varepsilon>0$, we have that $\vert  \sigma^n(V_0)\vert < \varepsilon n^k\theta^n$ for all $n$ large enough, 
since by construction $V_0$ contains no letter with maximal growth.  
Choosing $\varepsilon < c_1/2$, we then infer from (\ref{eq:c2'}) that 
$$
\frac{\vert\sigma^{n+n_0+m_0}(a) \vert}{(n+n_0+m_0)^k\theta^{n+n_0+m_0}}  <  c_1\, ,
$$
as soon as $n$ is large enough. This provides a contradiction with (\ref{eq:c2}).  

\medskip

Since there are infinitely many occurrences in ${\bf u}$ of 
letters with maximal growth, the pigeonhole principle ensures the existence of such a letter $b$ 
that occurs at least twice in ${\bf u}$.  
In particular, there exist two possibly empty finite words $U$ and $V$ such that $UbVb$ is a prefix of $\bf u$. 
Set $r=\vert U\vert$, $s=\vert bV\vert$, and for every nonnegative integer $n$, 
$U_n:=\sigma^{n}(U)$, $V_n:=\sigma^n(bV)$.  Since by definition ${\bf u}$ is fixed by $\sigma$, we get that $U_nV_n^{\delta_n}$ 
is a prefix of ${\bf u}$, where 
$\delta_n:= 1+\vert \sigma^n(b)\vert/\vert \sigma^n(bV)\vert$. 
Since $b$ has maximal growth, there exists a positive real number $c_3$ such that 
$$
\vert \sigma^n(c)\vert < c_3\vert \sigma^{n}(b)\vert \, , 
$$
for every letter $c$ in ${\bf u}$. 
We thus obtain that 
$$
\frac{\vert U_nV_n^{\delta_n}\vert}{\vert U_nV_n\vert} \geq 
1 + \frac{\vert \sigma^n(b)\vert}{\vert \sigma^n(UbV)\vert} \geq 1 + \frac{1}{c_3(r+s)} >1.
$$
This proves that $\Dio({\bf u})>1$. By definition of the output sequence produced by $\mathcal T$, one has  
${\bf a} := \varphi({\bf u}) $.  
It thus follows that $\Dio({\bf a}) \geq \Dio({\bf u}) >1$, for applying a coding to an infinite word cannot decrease the Diophantine exponent. 
This ends the proof. 
\end{proof}

\begin{proof}[Proof of Theorem \ref{thm:2bis}] The result follows directly from Propositions ABL and~\ref{prop:diotag}. 
\end{proof}

%%%%%%%%%%%%%%%%%%%%%%%%%%%%%%%
\subsection{Proof of theorem~\ref{thm:1bis}}\label{sec:proof1bis}

In order to prove Theorem~\ref{thm:1bis}, we first  introduce a useful and natural equivalence relation on the set of internal configurations of a pushdown automaton. 
This equivalence relation is closely related to the classical Myhill-Nerode relation used in formal language theory.  
Roughly, we think about two configurations as being equivalent if, starting from each configuration, 
there is no way to distinguish them by feeding the machine 
with arbitrary inputs.

\medskip

Let us introduce  some notation. An internal configuration of a $k$-{pushdown automaton} $\M =(Q,\Sigma_k,\Gamma,\delta,q_0,\Delta,\tau)$ is  a pair 
$C=(q,W)\in Q\times \Gamma^*$ where $q$ denote the state of the finite control and $W$ denote the word written on the stack.  
Given an input word $w$, $C_{\mathcal M}(w)$, or for short $C(w)$ if there is no risk of confusion, 
will denote the internal configuration reached by the machine $\mathcal M$ 
when starting from the initial configuration and feeding it with the input $w$, that is  $C(w)=\delta(q_0,\#,w)$. By the way,  
$\tau(C(w))$ will denote the corresponding output symbol produced by $\mathcal M$. 
We will also use the classical notation 
$C \xvdash{w} C'$ to express that starting from the internal configuration $C$ and reading the input word $w$, 
the machine enters into the internal configuration $C'$:  $\delta(C,w)=C'$

When the input alphabet is $\Sigma_k$ and $n$ is a 
natural number, we will simply write $C(n)$ instead of $C(\langle n\rangle_k)$. 
%We will say that such internal configurations are accessible. 

\begin{defn}\label{def:equiv}
Let $\mathcal M$ be a $k$-pushdown automaton. 
Given two input words $x$ and $y$, we say that $C(x)$ and $C(y)$ are equivalent, 
and we note $C(x)\sim C(y)$ if:

for every input $w$, if  $C(x) \xvdash{w} C_1$ and $C(y)\xvdash{w}C_2$, one has $\tau(C_1)=\tau(C_2)$
 
\end{defn}

It is obvious that $\sim$ is an equivalence relation. 
We are now ready to state the following simple but key result. 

\begin{prop}\label{prop:equiv} 
Let $\xi$ be a real number generated by 
a  $k$-pushdown automaton. Let us assume that the equivalence relation $\sim$ 
is nontrivial in the sense that there exist 
two distinct positive integers $n$ and $n'$ such that $C(n)\sim C(n')$. 
Then $\xi$ is either rational or transcendental. 
\end{prop}

\begin{proof}
Let $\xi$ be a real number whose base-$b$ expansion can be generated by a 
a  $k$-pushdown automaton $\mathcal M$. Let us denote by ${\bf a} := (a_n)_{n\geq 1}$ 
the output sequence of $\mathcal M$, so that $\langle \{\xi\}\rangle_b = 0.a_1a_2\cdots$.   
Let us assume that there exist two positive integers $n$ and $n'$, $n<n'$, such that $C(n)\sim C(n')$.  
Set $w_n:= \langle n\rangle_k$ and $w'_n=\langle n'\rangle_k$. By definition of the equivalence relation, one has:
$$
a_{[w_nw]_k} = a_{[w'_nw]_k} \, ,
$$
for every word $w\in \Sigma_k^*$. 
Given a positive integer $\ell$, we obtain in particular the following equalities: 
\begin{equation}\label{eq:Nerode}
\forall i\in [0,k^{\ell}-1],\;\; a_{k^{\ell}n+i} = a_{k^{\ell}n'+i} \, .
\end{equation} 
Set $U_{\ell} := a_1a_2\cdots a_{k^{\ell}n-1}$ and $V_{\ell} := a_{k^{\ell}n}a_{k^{\ell}n+1}\cdots a_{k^{\ell}n'-1}$.  
We thus deduce from (\ref{eq:Nerode}) that the word 
$$
U_{\ell}V_{\ell}^{1+ 1/(n'-n)} := a_1a_2\cdots a_{k^{\ell}n-1}a_{k^{\ell}n}a_{k^{\ell}n+1}\cdots 
a_{k^{\ell}n'-1}a_{k^{\ell}n}\cdots a_{k^{\ell}n+k^{\ell}-1}
$$
is a prefix of ${\bf a}$.  Furthermore, one has 
$$
\vert U_{\ell}V_{\ell}^{1+ 1/(n'-n)}\vert/\vert U_{\ell}V_{\ell}\vert = 1 + \frac{1}{n'-1/k^l} \geq 1+ \frac{1}{n'-1}\, \cdot
$$
 Since the exponent $1+ 1/(n'-1)$ does not depend on $\ell$, 
this shows that 
$$
\Dio({\bf a}) \geq 1+ 1/(n'-1) >1 \, .
$$
Then Proposition ABL implies that $\xi$ is either rational or transcendental, which ends the proof. 
\end{proof}

Notice that, with this proposition in hand, we observe that Theorem AB becomes obvious. 
\begin{proof}[Proof of Theorem AB]
Indeed, a  finite $k$-automaton can be seen as a $k$-pushdown automaton with empty stack alphabet (transitions only depend on the state and do not act on the stack), so a configuration is just given by the state of the finite control, and the empty stack.  

Since there are only a finite number of states, a finite automaton has only a finite number of different 
possible configurations.  By the pigeonhole principle, there thus exist two distinct positive integers  
$n$ and $n'$ such that $C(n)=C(n')$. Then the proof follows from Proposition \ref{prop:equiv}. 
\end{proof}

%%%%%%%%%%%%%%
We our now ready to prove the Theorem~\ref{thm:1bis}.

\begin{proof}[Proof of Theorem ~\ref{thm:1bis}]
Let $\xi$ be a real number that can be generated by a $k$-pushdown automata, say  
$\mathcal M:=(Q,\Sigma_k,\Gamma,\delta,q_0,\Delta,\tau)$. 
Given an input word $w\in \Sigma_k^*$, we denote by $q_{w}$ the state reached by $\mathcal M$ when starting from its initial configuration 
and reading the input $w$. We also denote by $S(w)\in \Gamma^*$ the corresponding content of the stack of $\mathcal M$ and by 
$H(w)$ the corresponding stack height, that is the length of  the word $S(w)$.  
With this notation, we obtain that starting from the initial configuration $(q_0,\#)$ and reading the input $w$, $\mathcal M$ 
reaches the internal configuration$(q_w,S(w))$, that is $(q_0,\#) \xvdash{w} (q_w,S(w))$.

\medskip

Let us denote by $\mathcal R_k:= \left(\Sigma_k\setminus \{0\}\right)\Sigma_k^*$ the language of all proper base-$k$ expansion of 
positive integers  (written from most to least significant digit). Then for every positive integer $n$, there is a unique word $w$ 
in $\mathcal R_k$ such that $\langle n\rangle_k=w$. For every positive integer $m$, we consider the set 
$$
\mathcal H_m := \left\{w \in \mathcal R_k \mid H(w) \leq m \right\}. 
$$
We distinguish two cases. 

\medskip

(i) Let us first assume that there exists a positive integer $m$ such that $\mathcal H_m$ is infinite.  
Note that for all $w\in \mathcal H_m$, the configuration $C(w) = (q_w,S(w))$ belongs to the finite set 
$\Delta \times \Gamma^{\leq m}$, where $\Gamma^{\leq m}$ denotes the set of words of length at most $m$ defined over $\Gamma$. 
Since $\mathcal H_m$ is infinite, the pigeonhole principle ensures the existence of to distinct 
words $w$ and $w'$ in $\mathcal H_m$ such that $C(w)=C(w')$. Setting $n:=[w]_k$ and $n':=[w']_k$, we obtain that $n\not= n'$ and 
$C(n)=C(n')$. In particular, $C(n)\sim C(n')$. Then Proposition \ref{prop:equiv} applies, which concludes the proof in that case. 

\medskip

(ii) Let us assume now that all sets $\mathcal H_m$ are finite. 
For every $m\geq 1$, we can thus pick a word $v
_m$ in $\mathcal H_m$ with maximal length.  
Note that since $\mathcal H_m\subset \mathcal H_{m+1}$, we have $\vert v_m\vert \leq \vert v_{m+1}\vert$. 
Furthermore, one has 
$$
\mathcal R_k = \bigcup_{m=1}^{\infty} \mathcal H_m \, ,
$$ 
which implies that the set  $\left\{ v_m \mid m\geq 1\right\}$
is infinite. 

As discussed in Remark \ref{subsec:eps}, we can assume without loss of generality 
that all $\epsilon$-moves of $\mathcal M$  are decreasing ones.   
Furthermore, recall that all possible $\epsilon$-moves are effectively 
performed after reading the last symbol of a given input. 
This leads to the following alternative. For every internal configuration $(q_w,S(w))$,  
each time a new input symbol $a$ is consumed, one has:

\medskip

\begin{itemize}

  \item  Either the stack height is decreased, which means that $H(wa) < H(w)$. 
 
 \medskip
 
  \item  Or only the topmost symbol of the stack has been modified, which formally means that there exist two words 
  $X,Y\in \Gamma^*$ and a letter $z\in \Gamma$ such that $S(w)=Xz$ while $S(wa)=XY$.    
 
\end{itemize}

\medskip

The definition of $v_m$ ensures that 
\begin{equation}\label{eq:sw}
\forall w \in \Sigma_k^*,\,   \;\;\; H(v_m) < H(v_mw) \,  .
\end{equation} 
Furthermore, if $m$ is large enough, we have that $H(v_m) > 1$. 
For such $m$, let us decompose the stack word $S(v_m)$ as 
$$
S(v_m)=X_mz_m\, ,
$$ 
where $z_m\in \Gamma$ is the topmost stack symbol.  
Inequality (\ref{eq:sw}) implies that for all $w\in \Sigma_k^*$, the word $X_m$ is a prefix of the stack word $S(v_mw)$.  
In other words, the part of the stack corresponding to the word $X_m$ will never be modified or even read during the computation 
$(q_{v_m},S(v_m)) \xvdash{w} (q_{v_mw},S(v_mw))$.  
This precisely means that 
$$
(q_{v_m},S(v_m))\sim (q_{v_m},z_m) \, .
$$ 
Note that $(q_{v_m},z_m) \in \Delta \times \Gamma$, which is a finite set, while we already observed that 
$\left\{ v_m \mid m\geq 1\right\}$
is infinite. 
The pigeonhole principle thus implies the existence of two distinct integers $m$ and $m'$ such that 
$v_m\not=v_{m'}$ and $C(v_m) \sim C(v_{m'})$. Setting $n:=[v_m]_k$ and $n':=[v_{m'}]_k$, we get that  
$C(n)\sim C(n')$ and $n\not= n'$. Then Proposition \ref{prop:equiv} applies, which ends the proof. 
\end{proof}

%%%%%%%%%%%%%%%%%%%%%%%%%%%%%%%%%%%%%%%%%%%%%%%%%%%

%%%%%%%%%%%%%%%%%%%%%%%%%%%%%%%%%%%%%%%%%%%
\section{Some related models of computations : tag machines and stack machines}\label{sec:ts}

In this section, we complete our study by discussing different types of machines.  
We first consider the \emph{tag machine} that was originally introduced by Cobham \cite{Co68} 
and we prove that Theorem \ref{thm:2bis} is well equivalent to Cobham's second claim. Then we introduce a general model of computation 
called \emph{multistack machine} and we show how Theorem \ref{thm:1bis} allows us to solve the Hartmanis--Stearns problem for this class of machines. 

%%%%%%%%%%%%%%%%%%%%%
\subsection{Tag-machines}\label{sec: tag} 

Cobham originally investigated in \cite{Co68} a model of computation called \emph{tag machine} whose outputs 
turn out to be precisely the morphic sequences defined in Section \ref{sec: morphism}. We describe here this model and the associated 
notion of dilation factor, and prove 
the equivalence between sequences produced by a tag machine with dilation factor larger than one and sequences generated by 
a morphism with exponential growth. This shows that Theorem \ref{thm:2bis} is well equivalent to Cobham's second claim, as 
claimed in the introduction.

 \medskip
 
 A tag machine is a two-tape enumerator that can be described as follows. 
 In internal structure, a tag machine $\mathcal T$ has:
 
 \medskip

\begin{itemize}

\item[$\bullet$] A finite state control. 

\medskip

\item[$\bullet$] A tape on which operate a read-only  head $\mathfrak R$ and a write-only head $\mathfrak W$.  

\end{itemize}

\medskip

In external structure, $\mathcal M$ has:

\medskip

\begin{itemize}

\item[$\bullet$]  An output tape on which operates a write-only head $\mathfrak W'$ and from which nothing can be erased. 

\end{itemize}

\medskip

\begin{figure}[!h]
\begin{center}
\psfrag{a}{\small$b$}
\psfrag{s(a)}{\small$\sigma(b)$}
\psfrag{p(a)}{\small$\varphi(b)$}
\psfrag{R}{\small $\mathfrak R$}
\psfrag{W}{\small$\mathfrak W$}
\psfrag{W'}{\small$\mathfrak W'$}
\psfrag{Q}{$Q$}
\psfrag{input tape}{\small Input tape}
\psfrag{control units}{\small Finite control}
\psfrag{output tape}{\small Output tape}
\includegraphics[width= 7.5cm]{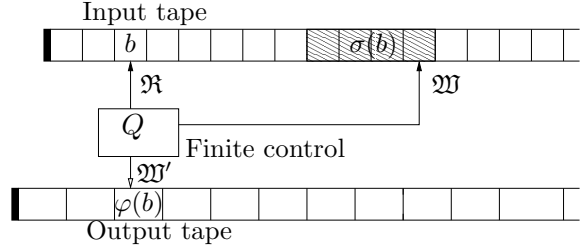}\caption{A tag machine}\label{fig:tagmachine}
\end{center}
\end{figure}

\medskip

Let us briefly describe how a tag machine operates. The finite state control of $\mathcal T$ contains some basic information: 
a finite set of symbols $A$ together with a special starting symbol $a$, so that with every element $b$ 
of $A$ is associated a finite word $\sigma(b)$ over $A$ and a symbol $\varphi(b)$ that belongs to a 
finite set of symbols $B$. 
When the computation starts, $\mathfrak R$ and $\mathfrak W$ are both positioned 
on the leftmost square of the (blank) tape and 
$\mathfrak W$ proceeds writing the word $\sigma(a)$, one symbol per square. Then both head $\mathfrak R$ and $\mathfrak W$ 
move one square right, $\mathfrak R$ scans the symbol written in the corresponding square, say $b$, and $\mathfrak W$ 
proceeds writing the word $\sigma(b)$. Again both heads move one square to the right and the process 
keeps on for ever unless $\mathfrak R$ eventually catches $\mathfrak W$ in which case the machine stops. Meanwhile, each time $\mathfrak R$ 
reads a symbol $b$ on the internal tape, $\mathfrak W'$ writes the symbol $\varphi(b)$ on the output tape and moves one square right.  
Each symbol written on the output tape is thus irrevocable and cannot be erased in the process of computation.  
The output sequence produced by $\mathcal T$ is the sequence of symbols written on its output tape. 

Using this description, Cobham extracts in \cite{Co68} the following usual definition of a tag machine in terms 
of morphisms which confirms that output of tag machines and morphic sequences are the same.   

\begin{defn}\label{def:tag}
{\rm A tag machine is a $5$-tuple ${\mathcal T} = (A,\sigma,a,B,\varphi)$ where:

\begin{itemize}

\medskip

  \item
 $A$ is a  finite set of symbols called the \emph{internal alphabet}. 
 
  \medskip
 
  \item 
$a$ is a an element of $A$ called the \emph{starting symbol}.
 
 \medskip
 
  \item 
$\sigma$ is a \emph{morphism} of $A^*$ prolongable on $a$.

\medskip

  \item 
 $B$ is a  finite set of symbols called the \emph{external alphabet}. 
 
  \medskip
 
  \item 
$\varphi$ is a letter-to-letter morphism from $A$ to $B$.

 \end{itemize}
 
 \medskip
 
 The output sequence of $\mathcal T$ is the morphic sequence $\varphi(\sigma^{\omega}(a))$.  
 A tag machine is said to be uniform (resp.\ $k$-uniform) when the morphism $\sigma$ has 
the additional property to be uniform (resp.\ $k$-uniform). }
\end{defn}

\medskip

In \cite{Co68}, Cobham also  introduced the following interesting  quantity which measures the rate of production of symbols by a tag machine.  

\begin{defn}{\rm
The ({\it minimum}) {\it dilation} factor of a tag machine $\mathcal T$ is defined by
$$
\mathfrak d(\mathcal T) = \liminf_{n\to\infty} \frac{\mathfrak W(n)}{n} \, ,
$$
where $\mathfrak W(n)$ denotes the position of the write-only head $\mathfrak W$ of $\mathcal T$ when the read-only head 
$\mathfrak R$ occupies the $n$-th square of the internal tape.} 
\end{defn}

\begin{rem}\label{rem:dilation}
It follows from Definition \ref{def:tag} that $\mathfrak W(n)=\vert\sigma(u_1u_2\cdots u_n)\vert$, where $u_1u_2\cdots u_n$ is the 
prefix of length $n$ of the infinite word $\sigma^{\omega}(a)$. 
\end{rem}

It is easy to see that uniform tag machines, or equivalently finite automata used as transducers 
(see Section \ref{sec: morphism}), all have dilation factor at least two.  As already mentioned,  
Cobham claimed that the base-$b$ expansion of an algebraic irrational number cannot be generated by a tag machine with dilation factor larger than one. 
This result will immediately follow from Theorem \ref{thm:2bis} once we will have proved Proposition  \ref{prop:equivtag} below. 
It  is stated without proof by Cobham in \cite{Co68}. 

\begin{prop}\label{prop:equivtag}
Let $\mathcal T:=(A,\sigma,a,B,\varphi)$ be a tag machine. Then the following statements are equivalent:

\medskip

\begin{itemize}

\item[\rm{(i)}]  $\mathfrak d({\mathcal T}) > 1$.

\medskip

\item[\rm{(ii)}]   The spectral radius of $M_{\sigma}$ is larger than one.

\end{itemize}
\end{prop}

\begin{proof}
Let us first prove that (i) implies (ii). Since $\mathfrak d(\mathcal T) > 1$, Remark \ref{rem:dilation} ensures the existence of 
a positive real number $\varepsilon$ such that 
$$
\frac{\vert \sigma^{n+1}(a)\vert}{\vert \sigma^n(a)\vert} = \frac{\mathfrak W(\vert \sigma^n(a)\vert)}{\vert \sigma^n(a)\vert} > 1+\varepsilon \,,
$$
for every $n$ large enough.  This implies that there exists a positive real number $c$ such that 
$$
\vert \sigma^n(a)\vert > c (1+\varepsilon)^n\, ,
$$
for every positive integer $n$. By Lemma \ref{lem:max}, we obtain that $\theta$, the spectral radius of $M_{\sigma}$, must satisfy  
$\theta \geq 1+\varepsilon>1$. 

\medskip

Let us now prove that (ii) implies (i).  Let $\theta>1$ denote the spectral radius of $M_{\sigma}$. We argue by contradiction assuming that 
$\mathfrak d(\mathcal T)=1$. 
Let ${\bf u}:=\sigma^{\omega}(a)$. 
By Lemma \ref{lem:max}, 
 there exist a nonnegative integer $k$,   and two positive real numbers $c_1$ and $c_2$ such that 
\begin{equation}\label{eq:theta3}
 c_1n^k\theta^n <\vert \sigma^n(a)\vert < c_2n^k\theta^n \, ,
\end{equation}
for every positive integer $n$.  
Set $C:=\Vert M_{\sigma}\Vert_{\infty}$. Let $\varepsilon$ be a positive number and let $m$ be a positive integer such that 
$$
\theta^m > C(1+\varepsilon)c_2/c_1 \,.
$$
We then infer from (\ref{eq:theta3}) that 
$$
\vert \sigma^{m+n}(a)\vert > c_1(m+n)^k\theta^{m+n} > \theta^m c_1n^k\theta^n> C(1+\varepsilon) c_2n^k\theta^n
$$
and thus
$$
\vert \sigma^{m+n}(a)\vert > C(1+\varepsilon) \vert \sigma^n(a)\vert\, ,
$$
for every positive integer $n$.  Let $N$ be a positive integer and let us denote by $u_1u_2\cdots u_{N}$ 
the prefix of length $N$ of ${\bf u}$. Let $n$ be the largest integer such that $\sigma^n(a)$ is a prefix of 
$u_1u_2\ldots u_{N}$. It thus follows that 
$$
\vert \sigma^m(u_1\cdots u_{N} )\vert \geq \vert \sigma^m(\sigma^n(a))\vert
 = \vert\sigma^{m+n}(a)\vert> C(1+\varepsilon) \vert \sigma^n(a)\vert \, .
$$
Since the definition of $n$ ensures 
that $\vert \sigma^n(a)\vert > N/C$, we have
\begin{equation}\label{eq:sigmam}
\vert \sigma^m(u_1u_2\cdots u_{N} )\vert > (1+\varepsilon) N \, .
\end{equation} 
On the other hand, for every $\delta>0$ there exists a positive integer $N$ such that:
$$
\frac{\vert \sigma(u_1u_2\cdots u_{N}) \vert}{N} < 1+\delta \, ,
$$
since by assumption $\mathfrak d(\mathcal T)=1$. 
Let $V$ be the finite word defined by the relation 
$\sigma(u_1u_2\cdots u_{N})=u_1u_2\ldots u_{N}V$. Thus $\vert V\vert < \delta N$. 
Now it is easy to see that 
$$
\sigma^m(u_1u_2\ldots u_{N})=u_1u_2\ldots u_{N}V\sigma(V)\cdots \sigma^{m-1}(V) \,,
$$
which implies that 
$$
\vert \sigma^m(u_1u_2\ldots u_{N})\vert < N + \delta N + C\delta N+\cdots + C^{m-1}\delta N \,.
$$
Choosing $\delta < \varepsilon(C-1)/(C^m-1)$, we get that $\vert \sigma^m(u_1u_2\ldots u_{N})\vert<(1+\varepsilon)N$, 
which contradicts (\ref{eq:sigmam}). This ends the proof. 
\end{proof}

%%%%%%%%%%%%%%%%%%%%%%%%%%%%%%%%%%%%
\subsection{Pushdown automata viewed as multistack machines}\label{subsec:mult}
%%%%%%%%%%%%%%%%%%%%%%%%%%%%%%%

In this section, we discuss how our main result allows us to revisit the Hartmanis--Stearns problem as follows.  
Instead of some time constraint,  we put some restriction based on the way the memory may be stored by Turing machines. 
We consider a classical  model of computation called {\it multistack machine}. It corresponds to a version of the deterministic Turing 
machine where the memory is simply organized by stacks. It is as general as the Turing machine if one allows two or more stacks. 
Furthermore the one-stack machine turns out to be equivalent to the deterministic pushdown automaton, while a zero-stack machine is 
just a finite automaton (a machine with a strictly finite memory only stored in the finite state control).

\medskip

For a formal definition of Turing machines the reader is referred to any of the classical references such as \cite{HMU,Mi,Si}. 
We will content ourself with the following informal definition of multistack machines.  
When used as a transducer, a multitape Turing machine can be divided into three parts:  

\begin{itemize}

\medskip

\item[$\bullet$] The input tape, on which there is a read-only head. 

\medskip

\item[$\bullet$]  The internal part, which consists in a finite control and the memory/working tapes (several tapes with one head per tape). 

\medskip

\item[$\bullet$]  The output tape on which there is a write-only head and from which nothing can be erased. 
\end{itemize}

\medskip

\begin{figure}[!h]
\begin{center}
\psfrag{a1}{\footnotesize$w_1$}
\psfrag{a2}{\footnotesize$w_2$}
\psfrag{a3}{\footnotesize$w_3$}
\psfrag{a4}{\footnotesize$w_4$}
\psfrag{a5}{\footnotesize$w_5$}
\psfrag{s1}{\footnotesize$s_1$}
\psfrag{s2}{\footnotesize$s_2$}
\psfrag{s3}{\footnotesize$s_3$}
\psfrag{s4}{\footnotesize$s_4$}
\psfrag{s5}{\footnotesize$s_5$}
\psfrag{s6}{\footnotesize$s_6$}
\psfrag{s7}{\footnotesize$s_7$}
\psfrag{q1}{\footnotesize$q_1$}
\psfrag{q2}{\footnotesize$q_2$}
\psfrag{q3}{\footnotesize$q_3$}
\psfrag{q4}{\footnotesize$q_4$}
\psfrag{q5}{\footnotesize$q_5$}
\psfrag{q6}{\footnotesize$q_6$}
\psfrag{r1}{\footnotesize$r_1$}
\psfrag{r2}{\footnotesize$r_2$}
\psfrag{r3}{\footnotesize$r_3$}
\psfrag{r4}{\footnotesize$r_4$}
\psfrag{Q}{$Q$}

\psfrag{b2}{\footnotesize$a(w)$}
\psfrag{input tape}{\footnotesize Input tape}
\psfrag{control units}{\footnotesize Finite control}
\psfrag{memory linear tapes}{\footnotesize Working tapes}
\psfrag{output tape}{\footnotesize Output tape}
\includegraphics[width= 9cm]{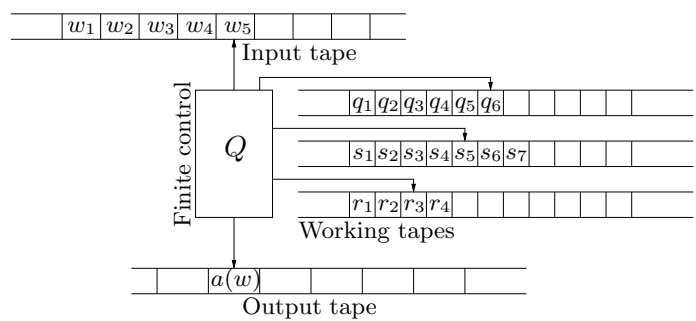}\caption{A multitape Turing machine}\label{fig:turingmachine}
\end{center}
\end{figure}

\medskip

Furthermore, the machine is said to be {\it one-way} or  {\it on-line} if the head of the input tape cannot go to the left. 
A (multi)stack machine is a one-way multi-tape deterministic Turing machine in which the memory is simply organized by stacks. 
This means that the head of each working/memory tape is always located on the rightmost symbol so that the tape can be though of simply as 
a stack with a head on the topmost symbol.

\begin{figure}[!h]
\begin{center}
\psfrag{a1}{\footnotesize$w_1$}
\psfrag{a2}{\footnotesize$w_2$}
\psfrag{a3}{\footnotesize$w_3$}
\psfrag{a4}{\footnotesize$w_4$}
\psfrag{a5}{\footnotesize$w_5$}
\psfrag{s1}{\footnotesize$s_1$}
\psfrag{s2}{\footnotesize$s_2$}
\psfrag{s3}{\footnotesize$s_3$}
\psfrag{s4}{\footnotesize$s_4$}
\psfrag{s5}{\footnotesize$s_5$}
\psfrag{s6}{\footnotesize$s_6$}
\psfrag{s7}{\footnotesize$s_7$}
\psfrag{q1}{\footnotesize$q_1$}
\psfrag{q2}{\footnotesize$q_2$}
\psfrag{q3}{\footnotesize$q_3$}
\psfrag{q4}{\footnotesize$q_4$}
\psfrag{q5}{\footnotesize$q_5$}
\psfrag{q6}{\footnotesize$q_6$}
\psfrag{r1}{\footnotesize$r_1$}
\psfrag{r2}{\footnotesize$r_2$}
\psfrag{r3}{\footnotesize$r_3$}
\psfrag{r4}{\footnotesize$r_4$}
\psfrag{Q}{$Q$}

\psfrag{b2}{\footnotesize$a(w)$}
\psfrag{input tape}{\footnotesize Input tape}
\psfrag{control units}{\footnotesize Finite control}
\psfrag{memory linear tapes}{\footnotesize Working stacks}
\psfrag{output tape}{\footnotesize Output tape}
\includegraphics[width= 9cm]{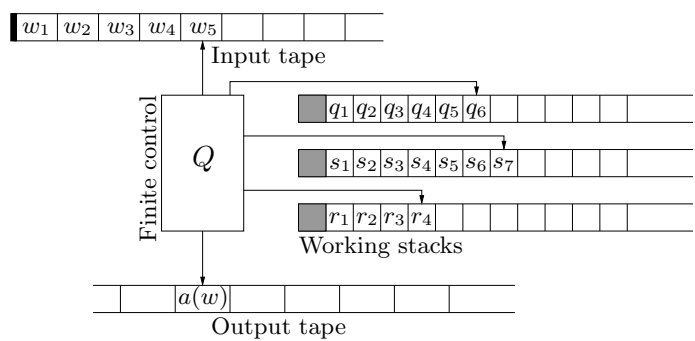}\caption{A one-way stack machine}\label{fig:stackmachine}
\end{center}
\end{figure}

Let us briefly describe how such a machine operates.  A move on a multistack machine $\mathcal M$ is based on:

\begin{itemize}

\medskip

  \item[$\bullet$]
 The current state of the finite control.
 
 \medskip
 
 \item[$\bullet$] The input symbol read. 
 
 \medskip
 
 \item[$\bullet$] The top stack symbol on each of its stacks. 

\end{itemize}

\medskip

Based on these data, a move of the multistack machine consists in:

\begin{itemize}

\medskip

  \item[$\bullet$] Change the finite state control to a new state.
  
   \medskip
 
 \item[$\bullet$] For each stack, replace the top symbol by a (possibly empty) string of stack symbols.   
 The choice of this string of symbols only depends on the input symbol read, the state of the finite control and on 
 the top symbol of each stack.
 
 \medskip
 
 \item[$\bullet$] Move the head of the input tape to the right.

\end{itemize}

\begin{rem}\label{rem:eps} Again, to make the machine deterministic, a move is uniquely determined by the knowledge of the input symbol read, 
 the state of the finite control and on 
 the top symbol of each stack.

As for pushdown automata, a multistack machine $\mathcal M$ can also possibly perform an $\epsilon$-move: a move for which the head of the input tape does not move. 
The possibility of such a move depends only on the current state of the finite control and the  top stack symbol on each of the stacks.  

  After reading a symbol of an input word $w$, the finite state control of 
 $\mathcal M$ could have reached a state $q$ from which $\epsilon$-moves are still possible. In that case, 
 we ask $\mathcal M$ 
to perform all possible $\epsilon$-moves before reading the next input symbol. Also,  
 a multistack machine is not allowed to stop its computation in a state from which an $\epsilon$-move is possible. 
\end{rem}

After reading an input word $w$,  
%\footnote{Formally, the word placed on the input tape should be $w\$$, where $\$$ is 
%a special symbol. This would allow the machine to know when the reading of the input ends. However, this subtlety is of 
%no importance for our purpose.}, 
$\mathcal M$ produces a output symbol $a(w)$ that belongs to a finite output alphabet.  
The symbol $a(w)$ depends only on the state of the finite control and the top symbol of each stack. Given an integer $k\geq 2$, 
a $k$-multistack machine is a multistack machine that takes as input the base-$k$ expansion of an integer (that is, for which the input alphabet is 
$\Sigma_k$). In that case, the sequence $a(\langle n\rangle_k)_{n\geq 0}$ is called the {\it output sequence} produced by $\mathcal M$.  
With these definition, a deterministic $k$-pushdown automata is nothing else than a $k$-multistack machine with a single stack.

\begin{figure}[!h]
\begin{center}
\psfrag{a1}{\footnotesize$w_1$}
\psfrag{a2}{\footnotesize$w_2$}
\psfrag{a3}{\footnotesize$w_3$}
\psfrag{a4}{\footnotesize$w_4$}
\psfrag{a5}{\footnotesize$w_5$}
\psfrag{s1}{\footnotesize$s_1$}
\psfrag{s2}{\footnotesize$s_2$}
\psfrag{s3}{\footnotesize$s_3$}
\psfrag{s4}{\footnotesize$s_4$}
\psfrag{s5}{\footnotesize$s_5$}
\psfrag{s6}{\footnotesize$s_6$}
\psfrag{s7}{\footnotesize$s_7$}
\psfrag{q1}{\footnotesize$q_1$}
\psfrag{q2}{\footnotesize$q_2$}
\psfrag{q3}{\footnotesize$q_3$}
\psfrag{q4}{\footnotesize$q_4$}
\psfrag{q5}{\footnotesize$q_5$}
\psfrag{q6}{\footnotesize$q_6$}
\psfrag{r1}{\footnotesize$r_1$}
\psfrag{r2}{\footnotesize$r_2$}
\psfrag{r3}{\footnotesize$r_3$}
\psfrag{r4}{\footnotesize$r_4$}
\psfrag{Q}{$Q$}

\psfrag{b2}{\footnotesize$a(w)$}
\psfrag{input tape}{\footnotesize Input tape}
\psfrag{States}{\footnotesize $\stackrel{\mbox{Finite}}{\mbox{control}}$}
\psfrag{stack}{\footnotesize Stack}
\psfrag{output tape}{\footnotesize Output tape}
\includegraphics[width= 9cm]{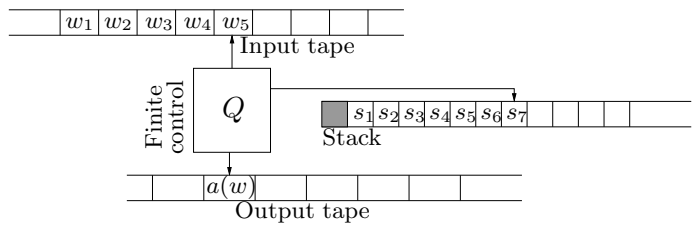}\caption{A deterministic pushdown automaton viewed as one-stack machine}\label{fig:PDAmachine}
\end{center}
\end{figure}

\subsubsection{The Hartmanis--Stearns problem for stack-machines} 
One can now define the class of real numbers generated by multistack machines as follows.  

\begin{defn} A real number $\xi$ can be generated by a $k$-multistack machine $\mathcal M$ if, 
for some integer $b\geq 2$, one has $\langle\{\xi\}\rangle_b=0.a_1a_2\cdots$, where $(a_n)_{n\geq 1}$  
corresponds to the output sequence produced by $\mathcal M$. 
A real number can be generated by a multistack machine if it can be generated by a $k$-multistack machine for some $k$. 
\end{defn}

Theorem \ref{thm:1bis} (resp.\ Theorem AB) can now be rephrased as follows: {\it no algebraic irrational can be generated by a one-stack (resp.\ zero-stack) machine}. 
Incidentally, this result turns out to provide a complete picture concerning the Hartmanis-Stearns problem for multistack machines. 
Indeed, since the two-stack machine has the same power as the general Turing machine, 
any computable number (and in particular any algebraic number) can be generated by a two-stack machine.

\subsection{Beyond pushdown automata}\label{sec:beyond}  

We stress now that the equivalence relation $\sim$ given in Definition \ref{def:equiv} and 
the associated Proposition \ref{prop:equiv} can be naturally extended to more general models of computation. 
For the multitape Turing machine, an internal configuration is determined by the state of the finite control and 
the complete knowledge of all the memory/working tapes (that is, the word written on each tape and the 
position of each head). But we do not need to concretely describe how the 
memory/working part of the machine is organized (tapes, stacks, or whatever). All what we need is to work with a machine 
with a one-way input tape and an output tape on which every symbol written is irrevocable. We  refer to this kind of machines as {\it one-way transducer-like machines}. 
The  Myhill-Nerode equivalence relation $\sim$ defined in~\ref{def:equiv} can be defined over the configurations of  a one-way transducer-like machine:  
two configurations being equivalent if, starting from each configuration, 
there is no way to distinguish them by feeding the machine 
with arbitrary inputs. Moreover Proposition~\ref{prop:equiv} also holds for these machines. 

\begin{prop}\label{prop:equiv-machines} 
Let $\xi$ be a real number generated by 
a one-way transducer like machine. Let us assume that the equivalence relation $\sim$ 
is nontrivial in the sense that there exist 
two distinct positive integers $n$ and $n'$ such that $C(n)\sim C(n')$. 
Then $\xi$ is either rational or transcendental. 
\end{prop}

This general result provides a method to prove the transcendence of real numbers which can be generated by machines more 
powerful than a 
$k$-pushdown automata. For instance, it implies the transcendence of the ternary number 
$$ \xi_3:= 0.1101201100101101001011001101001 10010110011010010112\dots $$
whose $n$-th ternary digit is equal to $2$ if the binary expansion of $n$ is of the form $1^k0^k1^k$, for some $k\in\mathbb{N}^*$, 
to $1$ if the binary expansion of $n$ has a odd number of occurrences of ones, and to $0$ otherwise. 
This number cannot be generated by a $k$-pushdown automaton because the set of words of the form $1^k0^k1^k$, for some $k\in\mathbb{N}^*$ is not a context-free language.  However, this language is context-sensitive which implies that $\xi_3$ can be generated by some kind of one-way transducer like machine (a linear bounded automaton). 
Proposition~\ref{prop:equiv-machines} can then be used to prove that $\xi_3$ is transcendental for one can show that $C(10)\sim C(20)$  
for every one-way transducer-like machine generating $\xi_3$.

%%%%%%%%%%%%%%%%%%%%%%%%%%%%%%%%%%%%%%%%%%
\section{Concluding remarks}\label{sec:end}

We end this paper with several comments concerning factor complexity, transcendence measures, and continued fractions, 
also providing  possible directions for further research.

\subsection{Links with factor complexity}\label{sec:comp} 

Another interesting way to tackle problems concerned with 
the expansions of classical constants in integer bases  
is to consider the {\it factor complexity} of real numbers. 
Let $\xi$ be a real number, $0\leq \xi<1$, and $b\geq 2$ be a positive integer.  Let us denote by 
${\bf a}=(a_n)_{n\geq 1}\in\Sigma_b^{\mathbb N}$ its base-$b$ expansion.  
The complexity function of $\xi$ with respect to the base $b$ is the function that associates  
with each positive integer $n$ the positive integer  
$$
p(\xi,b,n) := \mbox{Card} \{(a_j,a_{j+1},\ldots, a_{j+n-1}), \; j \geq 1\}.
$$
When $\xi$ does not belongs to $[0,1)$, we just set $p(\xi,b,n) := p(\{\xi\},b,n)$.  

To obtain lower bounds for the complexity of classical mathematical constants remains 
a famous challenging problem. In this direction, the main result concerning algebraic numbers was obtained by 
Bugeaud and the first author \cite{AB07} who proved that 
\begin{equation}\label{eq: complexity}
\lim_{n\to \infty} \frac{p(\xi,b,n)}{n} = + \infty \, ,
\end{equation}
for all algebraic irrational numbers $\xi$ and all integers $b\geq 2$.  
This lower bound implies Theorem AB for it is well-known that a real number generated by a finite automaton has factor complexity in $O(n)$ \cite{Co72}.   
We stress that the situation is really different with pushdown automata and tag machines. Indeed, given a positive integer $d$, there exist pushdown automata whose 
output sequence has a factor complexity 
growing at least like $n^d$ \cite{Mo08}, while tag machines can output sequences with quadratic complexity (see for instance \cite{Pa84}).  
In particular, Theorems \ref{thm:2bis} and \ref{thm:1bis} do not follow 
from  (\ref{eq: complexity}).   
We now exemplify this difference by providing lower bounds for the complexity of the 
two numbers $\xi_1$ and $\xi_2$ defined in Sections~\ref{sec: morphism} and~\ref{sec:PDA}.   

\subsubsection{A lower bound for $p(\xi_1,3,n)$}   It follows from the definition of the number $\xi_1$ that its ternary expansion is the fixed point of the morphism 
$\mu$ defined by $\mu(0)= 021$, $\mu(1)=012$, $\mu(2)=2$. We note that the letter $2$ has clearly bounded growth ($\vert \mu^n(2)\vert =1$ for all $n\geq 0$) 
and that $\mu^{\omega}(0)$ contains arbitrarily 
large blocks of consecutive occurrences of the letter $2$. Then, a classical result of Pansiot \cite{Pa84} implies that the complexity of the infinite word $\mu^{\omega}(0)$ is 
quadratic. In other words, one has : 
$$
c_1n^2 < p(\xi_1,3,n) < c_2n^2\,,
$$
for some positive constants $c_1$ and $c_2$.  
\subsubsection{A lower bound for $p(\xi_2,2,n)$}  
Recall that the binary number   $\xi_2$  
is defined as follows: its $n$-th binary digit  
is $1$ if the difference between the number of occurrences of the digits $0$ and $1$ in the binary expansion of $n$ is at 
most $1$, and is $0$ otherwise. 
We outline a proof of the fact that 
$$
p(\xi_2,2,n)=\Theta(n\log^2n) \, . 
$$ 

We can first infer from \cite{LG12} that 
$
p({\xi_2},2,n)=O(n(\log n)^2)$,  
for this sequence is generated by a pushdown automaton with only one ordinary stack symbol.  
In order to find a lower bound  for $p({\xi_2},2,n)$, we are going to describe a tag machine-like process (over an infinite alphabet) generating the binary expansion of $\xi_2$. 
We first notice that another way to understand the action of the $2$-PDA $\mathcal A$ in Figure \ref{fig:AP} that generates the binary expansion of $\xi_2$ is to unfold it. 
This representation, given in  Figure~\ref{fig:APunfolded}, corresponds to the transition graph  of $\mathcal A$: states in this graph are given by all possible configurations  and transitions between configurations are just labelled by the input digits $0$ or $1$. 
 In  Figure~\ref{fig:APunfolded}, the notation $qX^n$ means that, in this configuration, $\mathcal A$ is in state $q$ and the content of the stack is $XX\cdots X$ ($n$ times). 
\begin{figure}[!h]
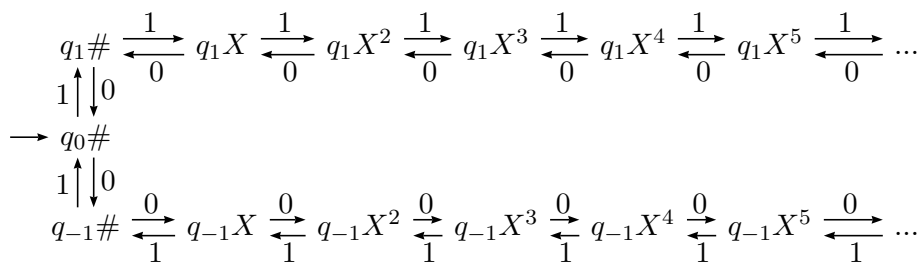

\begin{center}
\VCDraw{%

 \begin{VCPicture}{(-2,-3)(19,2.5)}
% states
\HideState
\StateVar[...]{(18,2)}{q_1s}
\StateVar[...]{(18,-2)}{rs}

%\LargeState
% \StateVar[s\#]{(-2,2)}{sE}
 \StateVar[q_1\#]{(0,2)}{q_1E} 
 \StateVar[q_1X]{(3,2)}{q_1X} 
 \StateVar[q_1X^2]{(6,2)}{q_1X2} 
 \StateVar[q_1X^3]{(9,2)}{q_1X3}
 \StateVar[q_1X^4]{(12,2)}{q_1X4} 
 \StateVar[q_1X^5]{(15,2)}{q_1X5} 
 \StateVar[q_0\#]{(0,0)}{q_0E}
\StateVar[q_{-1}\#]{(0,-2)}{rE} 
 \StateVar[q_{-1}X]{(3,-2)}{rX}
 \StateVar[q_{-1}X^2]{(6,-2)}{rX2} 
 \StateVar[q_{-1}X^3]{(9,-2)}{rX3} 
 \StateVar[q_{-1}X^4]{(12,-2)}{rX4}
 \StateVar[q_{-1}X^5]{(15,-2)}{rX5}

% initial--final
\Initial[w]{q_0E}
% transitions 
%\LoopS{sE}{0}

%\EdgeL{sE}{q_1E}{1}
\ForthBackOffset 
\EdgeL{q_1E}{q_1X}{1} \EdgeL{q_1X}{q_1E}{0}
\EdgeL{q_1X}{q_1X2}{1} \EdgeL{q_1X2}{q_1X}{0}
\EdgeL{q_1X2}{q_1X3}{1} \EdgeL{q_1X3}{q_1X2}{0}
\EdgeL{q_1X3}{q_1X4}{1} \EdgeL{q_1X4}{q_1X3}{0}
\EdgeL{q_1X4}{q_1X5}{1} \EdgeL{q_1X5}{q_1X4}{0}
\EdgeL{q_1X5}{q_1s}{1} \EdgeL{q_1s}{q_1X5}{0}

\EdgeL{q_1E}{q_0E}{0}
\EdgeL{q_0E}{q_1E}{1} 
\EdgeL{q_0E}{rE}{0}
\EdgeL{rE}{q_0E}{1} 
\EdgeL{rE}{rX}{0} \EdgeL{rX}{rE}{1}
\EdgeL{rX}{rX2}{0} \EdgeL{rX2}{rX}{1}
\EdgeL{rX2}{rX3}{0} \EdgeL{rX3}{rX2}{1}
\EdgeL{rX3}{rX4}{0} \EdgeL{rX4}{rX3}{1}
\EdgeL{rX4}{rX5}{0} \EdgeL{rX5}{rX4}{1}
\EdgeL{rX5}{rs}{0} \EdgeL{rs}{rX5}{1}

\end{VCPicture}
}

\caption{The transition graph of $\A$}\label{fig:APunfolded}
\end{center}
\end{figure}

In Figure~\ref{fig:APsubstitution},  states of the transition graph has been renamed as follows:  configurations are replaced with integers, where 
reading a $1$ in state $n$ leads to a move to state $n+1$ and reading a $0$ in state $n$ leads to a move to state $n-1$. 
We easily see that the output state is just the difference between the number of $1$'s and $0$'s in the input word.  Thus the $n$-th binary digit of $\xi_2$ is equal to $1$  
 if and only if the reading of the binary expansion of $n$ by this infinite automaton ends in one of the three states labelled by
  $0$, $-1$ and $1$.  

\begin{figure}[!h]
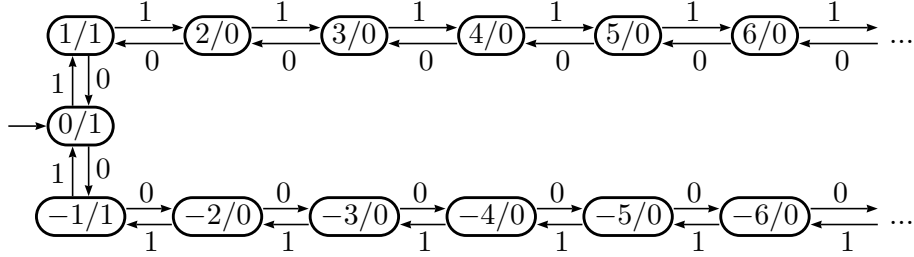

\begin{center}

\VCDraw{%

 \begin{VCPicture}{(-2,-3)(19,2.5)}
% states
\HideState
\StateVar[...]{(18,2)}{ps}
\StateVar[...]{(18,-2)}{rs}
\ShowState
% \StateVar[s/1]{(-2,2)}{sE}
 \StateVar[1/1]{(0,2)}{pE} 
 \StateVar[2/0]{(3,2)}{pX} 
 \StateVar[3/0]{(6,2)}{pX2} 
 \StateVar[4/0]{(9,2)}{pX3}
 \StateVar[5/0]{(12,2)}{pX4} 
 \StateVar[6/0]{(15,2)}{pX5} 
 \StateVar[0/1]{(0,0)}{qE}
\StateVar[-1/1]{(0,-2)}{rE} 
 \StateVar[-2/0]{(3,-2)}{rX}
 \StateVar[-3/0]{(6,-2)}{rX2} 
 \StateVar[-4/0]{(9,-2)}{rX3} 
 \StateVar[-5/0]{(12,-2)}{rX4}
 \StateVar[-6/0]{(15,-2)}{rX5}

% initial--final
\Initial[w]{qE}
% transitions 
%\LoopVarS{sE}{0}
%\EdgeL{sE}{pE}{1}
\ForthBackOffset 
\EdgeL{pE}{pX}{1} \EdgeL{pX}{pE}{0}
\EdgeL{pX}{pX2}{1} \EdgeL{pX2}{pX}{0}
\EdgeL{pX2}{pX3}{1} \EdgeL{pX3}{pX2}{0}
\EdgeL{pX3}{pX4}{1} \EdgeL{pX4}{pX3}{0}
\EdgeL{pX4}{pX5}{1} \EdgeL{pX5}{pX4}{0}
\EdgeL{pX5}{ps}{1} \EdgeL{ps}{pX5}{0}

\EdgeL{pE}{qE}{0}
\EdgeL{qE}{pE}{1} 
\EdgeL{qE}{rE}{0}
\EdgeL{rE}{qE}{1} 
\EdgeL{rE}{rX}{0} \EdgeL{rX}{rE}{1}
\EdgeL{rX}{rX2}{0} \EdgeL{rX2}{rX}{1}
\EdgeL{rX2}{rX3}{0} \EdgeL{rX3}{rX2}{1}
\EdgeL{rX3}{rX4}{0} \EdgeL{rX4}{rX3}{1}
\EdgeL{rX4}{rX5}{0} \EdgeL{rX5}{rX4}{1}
\EdgeL{rX5}{rs}{0} \EdgeL{rs}{rX5}{1}

\end{VCPicture}
}

\caption{Relabelling of the transition graph of $\A$}\label{fig:APsubstitution}
\end{center}
\end{figure}

The action of $0$ and $1$ can be summarized by $n \xrightarrow[]{0}n-1$ and  $n \xrightarrow[]{1}n+1$. This leads to a  tag machine-like process over an infinite alphabet 
$\mathcal{T}=(A,\sigma, s ,B,\varphi)$  for generating the expansion of $\xi_2$.  
The starting symbol is $s$, $A=\mathbb{Z}\cup\{s\}$, $\sigma$ is defined by $\sigma(s)=s1$ and $\sigma(n)=(n-1)(n+1)$, 
$B=\{0,1\}$, $\varphi(-1)= \varphi(0)= \varphi(1)= 1$, and $\varphi(e)= 0$ if $e\notin\{s,-1,0,1\}$. Then we have $\langle \xi\rangle_2 =0.\varphi(\sigma^\omega(s))$, 
where the infinite word 
$$\sigma^\omega(s)=s\,1 \,0 \,2 \,(-1)\,1\,1\,3\,(-2)\,0\,0\,2\,0\,2\,2\,4\,(-3)\,(-1)\,(-1)\,1\,(-1)\,1\,1\,3\,(-1)\,1\,1\cdots$$ 
is the unique fixed point of the morphism $\sigma$.   
%For $k\in\mathbb N$ and  $x\in\mathbb Z$, we note that the word $\varphi(\sigma^k(x))$ is almost always equal to $0^{2^k}$.  
%More precisely, $\sigma^k(x)$ contains an occurrence of $-1$, $0$, or $1$ if and only if $x\in [-(k+1), k+1]$. 
%Given a positive integer $n$, there exists a unique $k$ such that $2^k\leq n < 2^{k+1}$. 

%Furthermore,  
%any factor of length $n$ of $\sigma^\omega(s)$  is a factor of some $\sigma^k(x)\sigma^k(y)$, where $xy$ is a factor of length two of $\sigma^\omega(s)$ and thus   
%any factor of length $n$ of $\varphi(\sigma^\omega(s))$ is a factor of some $\varphi(\sigma^k(x)\sigma^k(y))$. This gives that  
%$$
%p(\xi_2,2,n)=O(n\log^2n)\, .
%$$ 
The strategy consists now in finding sufficiently many different right special factors, that is factors 
$w$ of $\varphi(\sigma^\omega(s))$ for which both factors $w0$ and $w1$ also occur in $\varphi(\sigma^\omega(s))$.  
% Any right special factor of $\varphi(\sigma^\omega(s))$ of length $n$ is contained in  
%some $\varphi(\sigma^k(x)\sigma^k(y)\sigma^k(z))$ and $\varphi(\sigma^k(x)\sigma^k(y)\sigma^k(z'))$,  
%where both words $xyz$ and $xyz'$ are  factors of  $\sigma^\omega(s)$. 
Arguing  as in ~\cite[Lemma 1.13]{LG08}, one can actually show that, for every pair 
$$
(p,q)\in {\mathcal E}:=\{(p,q)\in\mathbb{N}^2 \mid 1\le p\le q\le k-2\} \,,
$$ 
both words 
$$A:=\varphi(\sigma^k(k-2p)\sigma^k(k-2p+2)\sigma^k(k-2q))$$ 
and  
$$B:=\varphi(\sigma^k(k-2p)\sigma^k(k-2p+2)\sigma^k(-k-2))$$ 
occur in $\varphi(\sigma^\omega(s))$ and they have the same factor of length $n$, say $w(p,q)$,  
occurring at index $2^{k+1}+2^{q}-n$.   Furthermore, in the word $A$ the factor $w(p,q)$ is followed by a $1$, while in the word $B$ it is followed by a $0$.  
 Thus  $w(p,q)$ is  a right special factor.  It can also be extracted from ~\cite[Lemma 1.13]{LG08} that  the map $(p,q)\mapsto w(p,q)$ is injective on $\mathcal E$.  
This ensures the existence of at least  $ \frac{(k-2)(k-1)}{2}$ distinct right special factors of length $n$ in $\varphi(\sigma^\omega(s))$.   
Then it follows that 
$$
p(\xi_2,2,n+1) - p(\xi_2,2,n)\ge   \frac{(k-2)(k-1)}{2} \, ,
$$
from which one easily  deduces the lower bound :  
$$
p({\xi_2},2,n)\ge cn(\log n)^2\, ,
$$ 
for some positive constant $c$.

%%%%%%%%%%%%%%
\subsection{Quantitative aspects: transcendence measures and the imitation game}\label{sec:quant}  
We discuss here some problems related to the quantitative aspects of our results. 

\subsubsection{The number theory side: transcendence measures}  
A real number $\xi$ is transcendental if $\vert P(\xi)\vert >0$, for all non-zero integer polynomials $P(X)$.   A transcendence measure 
for $\xi$ consists in a limitation of the smallness of $\vert P(\xi)\vert$, thus refining the transcendence statement. In general,  
one looks for a nontrivial function $f$ satisfying: 
$$
\vert P(\xi)\vert > f(H,d) \, ,
$$ 
for all integer polynomials of degree at most $d$ and height at most $H$. 
Here, $H(P)$ stands for the na\"\i ve height of the
polynomial $P(X)$, that is, the maximum of the absolute values of
its coefficients. The degree and the height of an integer polynomial $P$ allow to take care of the complexity 
of $P$.   We will use here the following classification 
of real numbers defined by Mahler \cite{MahCr32} in 1932. 
For every integer $d \ge 1$ and every real number $\xi$, we denote by
$w_d (\xi)$ the supremum of the exponents $w$ for which
$$
0 < |P(\xi)| < H(P)^{-w}
$$
has infinitely many solutions in integer polynomials $P(X)$ of
degree at most $d$.  Further, we set 
$w(\xi) = \limsup_{d \to \infty} (w_d(\xi) /d)$ and,
according to Mahler \cite{MahCr32}, we say that $\xi$ is an
$$
\displaylines{
\hbox{ $A$-number, if $w(\xi) = 0$;} \cr
\hbox{ $S$-number, if $0< w(\xi) < \infty$;} \cr
\hbox{ $T$-number, if $w(\xi) = \infty $ and $w_d(\xi) < \infty$ for any
integer $d\ge 1$;} \cr
\hbox{ $U$-number, if $w(\xi) = \infty $ and $w_d(\xi) = \infty$ for some
integer $d\ge 1$.}\cr}
$$
An important feature of this classification is that two transcendental real   
numbers that belong to different classes are algebraically independent.  
The $A$-numbers are precisely the algebraic numbers and,
in the sense of the Lebesgue measure, almost all numbers are $S$-numbers.

 A Liouville number is a real number $\xi$ such that for any positive real number $\rho$ 
the inequality 
$$
\left\vert \xi -\frac{p}{q}\right\vert < \frac{1}{q^{\rho}}
$$
has a at least one solution $(p,q)\in\mathbb Z^2$, with $q>1$. Thus $\xi$ is a Liouville number if, and only if, $w_1(\xi)=+\infty$. 

\medskip

Let $\xi$ be an irrational real number defined through its base-$b$ expansion, say $\langle \{\xi\}\rangle_b :=  0.a_1a_2\cdots$. 
Let us assume that the base-$b$ expansion of $\xi$ can be generated either by a pushdown automata or by a tag machine with dilation factor larger than one. 
As recalled in  Proposition ABL (Section \ref{sec:ABL}), the key point for proving that $\xi$ is transcendental is to show that  $\Dio({\bf a})>1$, where ${\bf a}:= a_1a_2\cdots$.  
This comes down to finding two sequences of finite words $(U_n)_{n\geq 0}$ and $(V_n)_{n\geq 0}$, a sequence of rational numbers $\alpha_n$, 
and a real number $\delta>0$ such that 
the word $U_nV_n^{\alpha_n}$ is a prefix of ${\bf a}$, the length of the word $U_nV_n^{\alpha_n}$ increases, and 
\begin{equation}\label{in0}
\frac{\vert U_nV_n^{\alpha_n}\vert}{\vert U_nV_n \vert} \geq 1 + \delta \,.
\end{equation}
A look at the proofs of Theorems \ref{thm:2bis} and \ref{thm:1bis} show that one actually has, in both cases, the following extra property: there exists a real number $M$ such that 
\begin{equation}\label{eq:uniforme}
\limsup_{n\to\infty} \frac{\vert U_{n+1}V_{n+1}\vert}{\vert U_nV_n\vert} < M  \, .
\end{equation}
Using an approach introduced in \cite{AC06} and developed in \cite{AB11}, one can first prove that 
\begin{equation}\label{in1}
\Dio({\bf a})-1 \leq w_1(\xi) \leq c_1\Dio({\bf a}) \, ,
\end{equation}
for some real number $c_1$ that depends only on $\delta$ and $M$.  In particular, $\xi$ is a Liouville number if and only if $\Dio({\bf a})$ is infinite. 
Then  it is proved in \cite{AB11}, following a general approach introduced in \cite{AB10} and based on  a quantitative version of the subspace theorem,  
that this extra condition leads to transcendence measures. 
Indeed, taking all parameters into account, one could derive an upper bound of the type 
\begin{equation}\label{in2}
w_d(\xi) \leq \max\{w_1(\xi),(2d)^{c_2(\log 3d)(\log \log 3d)} \}\, ,
\end{equation}
for all positive integers $d$ and some real number $c_2$ that depends only on $\delta$ and $M$. The constants $c_1$ and $c_2$ can be made effective. 
In particular, we deduce from Inequalities \eqref{in1} and \eqref{in2} the following result.

\begin{thm}\label{thm:alt}
Let $\xi$ be an irrational real number such that $\langle \{\xi\}\rangle_b :=  0.a_1a_2\cdots$ and let ${\bf a}:= a_1a_2\cdots$. 
Let us assume that the base-$b$ expansion of $\xi$ can be generated either by a pushdown automata or by a tag machine with dilation factor larger than one. 
Then one of  the following holds.  
\begin{itemize}

\medskip

\item[{\rm (i)}]  $\Dio({\bf a})=+\infty$ and $\xi$ is a Liouville number. 

\medskip
 
 \item[{\rm (ii)}] $\Dio({\bf a})<+\infty$ and $\xi$ is a $S$- or a $T$-number. 

\end{itemize}
\end{thm}

Of course, in view of Theorem \ref{thm:alt}, it would be interesting to prove whether or not there exist such numbers for which $\Dio({\bf a})=+\infty$. 
In this direction, it is proved in \cite{AC06} that 
$\Dio({\bf a})$ is always finite when $\xi$ is generated by a finite automaton. We add here the following contribution to this problem.  

\begin{prop}
Let ${\bf a}:= a_1a_2\cdots$ be an aperiodic purely morphic word generated by a morphism $\sigma$ defined over a finite alphabet $A$. 
Set $M:= \max\{\vert \sigma(i)\vert \mid i\in A\}$.   
Then $\Dio({\bf a})\leq M+1$.
\end{prop}

\begin{proof}
Let us assume that $\sigma$ is prolongable on the letter $a$ and that 
$\sigma^{\omega}(a)=a_1a_2\cdots$. 
We argue now by contradiction by assuming that $\Dio({\bf a})>M+1$.

This assumption ensures that one can find  two finite words $U$ and $V$ and a real number $s>1$ such that :
\begin{itemize}

\item[{\rm (i)}] $UV^s$ is a prefix of ${\bf a}$, and $s$ is maximal with this property. 

\medskip

\item[{\rm (ii)}] One has 
$$\vert UV^s\vert/\vert UV\vert \geq M+1 \, .$$

\medskip

\item[{\rm (iii)}] $V$ is primitive (i.e.\ is non-empty and not the integral power of a shorter word) and  $s$ is maximal. 

%\medskip

%\item[{\rm (iv)}] $UV$ is long enough to ensure that $\vert\sigma(UV) \vert > \vert UV\vert$.  

\medskip
\end{itemize}

Not that since ${\bf a}$ is fixed by $\sigma$ then the word $\sigma(UV^s)$ is also a prefix of ${\bf a}$. 
By definition of $M$, it follows from (ii) that 
$$
UV^s = \sigma(U)W \, ,
$$
where  $W=\widetilde{V}^\alpha$ for some conjugate $\widetilde V$  of $V$ (i.e., $V=AB$ and $\widetilde{V}=BA$ for some $A,B$) and $\alpha\leq s$. 
On the other hand, $\sigma(V)$ is also a period of $W$ since $UV^s=\sigma(U)W$ is a prefix of  $\sigma(UV^s)=\sigma(U)\sigma(V)^{s'}$, for some $s'$. 
Thus $W$ has at least two periods: $\widetilde V$ and $\sigma(V)$. Furthermore, (ii) implies that 
$$
\vert UV^{s-1}\vert \geq M(\vert U\vert +\vert V\vert)
$$
and then 
$$
\vert W\vert = \vert UV^s\vert - \vert \sigma(U)\vert \geq \vert \sigma(V)\vert + \vert V\vert =  \vert \sigma(V)\vert + \vert \widetilde{V}\vert\, .
$$
We can thus apply Fine and Wilf's theorem (see for instance \cite[Chap. 1]{AS}) to the word $W$ and we obtain that there is a word of length 
$\gcd(\vert \widetilde{V}\vert,\vert\sigma(V)\vert)$ that is a period 
of $W$.  But, since $V$ is primitive, the word $\widetilde{V}$ is primitive too, and it follows that 
$\gcd(\vert \widetilde{V}\vert,\vert\sigma(V)\vert)=\vert\widetilde{V}\vert$. 
This gives that  
$\sigma(V)=\widetilde{V}^k$ for some positive integer $k$.  It follows that 
$$
\sigma(UV^s) = \sigma(U)\widetilde{V}^{ks'} = UV^{s-\alpha +ks'}
$$ 
is a prefix of ${\bf a}$. 
Now the inequality $\vert \sigma(UV^s)\vert > \vert UV^s\vert$ gives a contradiction with the maximality of $s$. 
This ends the proof. 
\end{proof}

%%%%%
\subsubsection{The computer science side: the imitation game}  
Theorems AB, \ref{thm:2bis}, and \ref{thm:1bis} show that some classes of Turing machines 
are too limited to produce the base-$b$ expansion of an algebraic irrational real number.  
Let $\xi$ be an irrational real number 
that can be generated by a $k$-pushdown automaton or by a tag machine with dilation factor larger than one. Then   
the results of Section \ref{sec:quant} could be rephrased  to provide a limitation of the way $\xi$ 
can be approximated by irrational algebraic numbers. 
In this section, we suggest to view things from a different angle,  changing our target. Indeed, we fix an algebraic irrational real number 
$\alpha$ and a base $b$, and ask for how long the base-$b$ expansion of $\alpha$  
can be imitated by outputs of a given class of Turing machines.   
%Note that though there do not exists a $k$-automaton able to produce the complete base-$b$ expansion of $\alpha$, 
%for every positive integer $N$, there are $k$-automata whose outputs coincides with ${\bf a}$ at least up to the $N$-th digit.  
%In other words, $k$-automatic sequences are dense with respect to the usual topology into the set of 
%infinite words.  

Let us explain now how to formalize our problem. We can naturally take the number of states as a measure of complexity of a $k$-automaton.   
 One can also defined the size of $k$-pushdown automata and 
tag machines as follows.  Let us define the size of a $k$-pushdown automaton $\mathcal A :=(Q,\Sigma_k,\Gamma,\delta,q_0,\Delta,\tau)$ to be 
$\vert Q\vert + \vert \Gamma\vert + L$, where $L$ is the maximal length of a word that can be added to the stack by the transition function $\delta$ of $\mathcal A$.
Let us also define the size of  a tag machine $\mathcal T:= (A,\sigma,a,\varphi,B)$ to be $\vert A\vert + L$, where $L:= \max \{\vert \sigma(i)\vert \mid i\in A\}$. 
Now, let us fix a class $\mathcal M$ of Turing machines among $k$-automata, $k$-pushdown automata, and tag machines.  Let $M$ be a positive integer. We stress that 
there are only finitely many such machines with size at most $M$.    
Then there exists 
a maximal positive integer $I(\alpha,M)$ for which  there exists a machine in $\mathcal M$ with size at most $M$  
whose output agrees with the base-$b$ expansion of $\alpha$ at least up to the $I(\alpha,M)$-th digit. 
We suggest the following problem. 

\begin{prob} 
Let $\alpha$ be an algebraic irrational real number and fix a class of Turing machines among $k$-automata, $k$-pushdown automata, and tag machines.   
Given a positive integer $M$, find an upper bound for $I(\alpha,M)$.  
\end{prob}

In the case of finite automata, we can give a first result toward this problem. Indeed, the factor complexity of the output ${\bf a}$ of a $k$-automaton with at most $M$ states satisfies  
 $p({\bf a}, n) \leq kM^2n$ (see for instance \cite{AS}). Let us denote respectively by $d$ and $H$ the degree and the height of $\alpha$. 
 Then the main result of \cite{Bu08} allows to extract the following upper bound : 
 \begin{eqnarray*}
 I(\alpha,M)&\leq& \max\left\{ (\max(\log H,e) 100kM^2)^{8\log 4kM^2} \right., 
 \\
&&\;\;\;\;\;\;\;\;\;\;\;\; \left. \left( (\log d) 10^{100} (kM^2)^{11/2}\log (kM^2)\right)^{2.1}\right\} \, .
 \end{eqnarray*}

%%%%%%%%%%%%%%%%%%%%%%%%%%%%%%%%%%%%%%%%%
\subsection{Computational complexity of the continued fraction expansion of algebraic numbers}\label{sec:cf}

Replacing  integer base expansions with continued fractions leads to similar problems. Rational numbers all have a finite continued fraction expansion, while quadratic 
real numbers correspond to eventually periodic continued fractions. In contrast, much less is known about the continued fraction expansion of algebraic real 
numbers of degree at least three such as $\sqrt[3]{2}$. 
In this direction, an approach based on the subspace theorem was introduced by Bugeaud and the first author \cite{AB05}. Recently, Bugeaud \cite{Bu13} 
shows that this approach actually leads to the following analogue of Proposition ABL. 

\medskip

\noindent {\bf\itshape Proposition {\rm\bf B}.} --- 
\emph{ Let $\xi$ be a real number with $\xi :=  [a_0,a_1,a_2,\ldots]$ where we assume that $(a_n)_{n\geq 1}$ is a bounded sequence of 
positive integers.  Let us assume that $\Dio({\bf a})>1$ where ${\bf a}:= a_1a_2\cdots$. 
 Then $\xi$ is either quadratic or transcendental. 
}

\medskip

In \cite{Bu13}, the author deduce from Proposition B that the continued fraction expansion of an algebraic real number of 
degree at least $3$ cannot be generated by a finite automaton. This provides the analogue of Theorem AB in this framework. 
As a direct consequence of our results and Proposition B, we obtain the following generalization of Bugeaud's result corresponding to the analogue of 
Theorems \ref{thm:2bis} and \ref{thm:1bis}.  

\begin{thm}
Let $\xi$ be an algebraic real number of degree at least $3$. Then the following holds. 
\begin{itemize}
\medskip

\item[{\rm (i)}]  The continued fraction expansion of $\xi$ cannot be generated by a one-stack machine, or equivalently, 
by a deterministic pushdown automaton. 
 
 \medskip
 
 \item[{\rm (ii)}] The continued fraction expansion of $\xi$ cannot be generated by a tag machine with dilation factor larger than one.

 \end{itemize}
\end{thm}

Using the approach introduced in \cite{AB12} and the discussion of Section \ref{sec:quant}, it will also be possible to produce transcendence measures analogous to Theorem \ref{thm:alt} 
for real numbers whose continued fraction expansion can be generated by deterministic pushdown automata or by a tag machine with dilation factor larger than one. 
%%%%%%%%%%%%%%%%%%%%%%%%%%%%%%%%%%%%%%%%%%

%%%%%%%%%%%%%%%%%%%%%%%%%%%%%%%%%%%%%%%%%

\end{document}